\documentclass[11pt]{amsart}
\usepackage[margin=1.2in]{geometry}

\usepackage{pinlabel}
         
\usepackage[all,cmtip]{xy}
\usepackage{amsmath,amssymb,amsthm, amscd, xypic}
\usepackage{mathtools}
\usepackage{mathbbol}
\usepackage[english]{babel}
\usepackage{tikz-cd, enumerate}
\usetikzlibrary{decorations.pathmorphing} 

\usepackage{mathrsfs}
\usepackage{fancyhdr}

\usepackage[normalem]{ulem}
\usepackage[all]{xy}

\usepackage{todonotes}
\usepackage{comment}

\usepackage{xcolor}

\usepackage{enumitem}

\colorlet{GREEN}{green}
\colorlet{BLUE}{blue}

\definecolor{darkgreen}{rgb}{0,0.50,0} 
\definecolor{darkred}{rgb}{0.55,0,0}
\definecolor{darkblue}{rgb}{0,0,0.6}

\usepackage[pdfborder=0,pagebackref,colorlinks,citecolor=darkgreen,linkcolor=darkred,urlcolor=darkblue]{hyperref}
\addto\extrasenglish{
    
}
\def\makeautorefname#1#2{\expandafter\def\csname#1autorefname\endcsname{#2}}
\numberwithin{equation}{section}
\makeautorefname{footnote}{footnote}%
\makeautorefname{item}{item}%
\makeautorefname{figure}{Figure}%
\makeautorefname{table}{Table}%
\makeautorefname{part}{Part}%
\makeautorefname{appendix}{Appendix}%
\makeautorefname{chapter}{Chapter}%
\makeautorefname{section}{Section}%
\makeautorefname{subsection}{Section}%
\makeautorefname{subsubsection}{Section}%
\makeautorefname{paragraph}{Paragraph}%
\makeautorefname{subparagraph}{Paragraph}%
\makeautorefname{theorem}{Theorem}%
\makeautorefname{thm}{Theorem}%
\makeautorefname{introthm}{Theorem}%
\makeautorefname{addm}{Addendum}%
\makeautorefname{mainthm}{Main theorem}%
\makeautorefname{corollary}{Corollary}%
\makeautorefname{cor}{Corollary}%
\makeautorefname{lemma}{Lemma}%
\makeautorefname{lem}{Lemma}%
\makeautorefname{sublemma}{Sublemma}%
\makeautorefname{sublem}{Sublemma}%
\makeautorefname{subl}{Sublemma}%
\makeautorefname{prop}{Proposition}%
\makeautorefname{property}{Property}
\makeautorefname{pro}{Property}
\makeautorefname{sch}{Scholium}%
\makeautorefname{step}{Step}%
\makeautorefname{conject}{Conjecture}%
\makeautorefname{conj}{Conjecture}%
\makeautorefname{questn}{Question}
\makeautorefname{quest}{Question}
\makeautorefname{qn}{Question}
\makeautorefname{definition}{Definition}%
\makeautorefname{defn}{Definition}%
\makeautorefname{defi}{Definition}%
\makeautorefname{def}{Definition}%
\makeautorefname{dfn}{Definition}%
\makeautorefname{notation}{Notation}
\makeautorefname{notn}{Notation}
\makeautorefname{rem}{Remark}%
\makeautorefname{rems}{Remarks}%
\makeautorefname{rmk}{Remark}%
\makeautorefname{rk}{Remark}%
\makeautorefname{remarks}{Remarks}%
\makeautorefname{rems}{Remarks}%
\makeautorefname{rmks}{Remarks}%
\makeautorefname{rks}{Remarks}%
\makeautorefname{example}{Example}%
\makeautorefname{examp}{Example}%
\makeautorefname{exmp}{Example}%
\makeautorefname{exam}{Example}%
\makeautorefname{exa}{Example}%
\makeautorefname{axiom}{Axiom}%
\makeautorefname{axi}{Axiom}%
\makeautorefname{ax}{Axiom}%
\makeautorefname{case}{Case}%
\makeautorefname{claim}{Claim}%
\makeautorefname{clm}{Claim}%
\makeautorefname{assumpt}{Assumption}%
\makeautorefname{asses}{Assumptions}%
\makeautorefname{conclusion}{Conclusion}%
\makeautorefname{concl}{Conclusion}%
\makeautorefname{conc}{Conclusion}%
\makeautorefname{cond}{Condition}%
\makeautorefname{const}{Construction}%
\makeautorefname{con}{Construction}%
\makeautorefname{criterion}{Criterion}%
\makeautorefname{criter}{Criterion}%
\makeautorefname{crit}{Criterion}%
\makeautorefname{exercise}{Exercise}%
\makeautorefname{exer}{Exercise}%
\makeautorefname{exe}{Exercise}%
\makeautorefname{problem}{Problem}%
\makeautorefname{problm}{Problem}%
\makeautorefname{prob}{Problem}%
\makeautorefname{prob}{Problem}%
\makeautorefname{soln}{Solution}%
\makeautorefname{sol}{Solution}%
\makeautorefname{sum}{Summary}%
\makeautorefname{oper}{Operation}%
\makeautorefname{obs}{Observation}%
\makeautorefname{ob}{Observation}%
\makeautorefname{conv}{Convention}%
\makeautorefname{cvn}{Convention}%
\makeautorefname{notation}{Notation}%
\makeautorefname{warn}{Warning}%
\makeautorefname{note}{Note}%
\makeautorefname{fact}{Fact}%
\makeautorefname{ouch0}{Counterexample}%
\newtheorem{thm}{Theorem}[section]

\newtheorem{prop}{Proposition}[section]

\newtheorem{lemma}{Lemma}[section]

\newtheorem{quest}{Question}[section]

\theoremstyle{definition}

\newtheorem{defn}{Definition}[section]

\newtheorem{con}{Construction}[section]

\newtheorem{exmp}{Example}[section]

\newtheorem{notation}{Notation}[section]

\newtheorem{rem}{Remark}[section]

\makeatletter
\let\c@cor=\c@thm
\let\c@prop=\c@thm
\let\c@proposition=\c@thm
\let\c@theorem=\c@thm
\let\c@lem=\c@thm
\let\c@definition=\c@thm
\let\c@conj=\c@thm
\let\c@defn=\c@thm
\let\c@def=\c@thm
\let\c@df=\c@thm
\let\c@exmp=\c@thm
\let\c@rem=\c@thm
\let\c@lemma=\c@thm
\let\c@quest=\c@thm
\let\c@sch=\c@thm
\let\c@con=\c@thm
\let\c@notation=\c@thm
\let\c@equation=\c@thm
\makeatother

\newcommand{\im}{\mathrm{Im}}

\newcommand{\End}{\mathrm{End}}

\newcommand{\SK}{\ensuremath{SK}}
\newcommand{\SKbar}{\overline{\ensuremath{SK}}}
\newcommand{\SKK}{\ensuremath{SKK}}
\newcommand{\id}{\mathrm{id}}

\def\C{\mathcal{C}}

\newcommand{\Mtheta}{\mathsf{Mfd}^\theta} 
\newcommand{\Biv}{\mathsf{Biv}}
\newcommand{\Cob}{\mathsf{Cob}}
\newcommand{\Cobf}{\mathsf{Cobf}}

\newcommand{\Cobfe}{\mathsf{Cobf}_{\varepsilon}}

\newcommand{\MTSO}{\mathrm{MTSO}}
\newcommand{\Gr}{\mathrm{Gr}}
\newcommand{\Th}{\mathrm{Th}}
\newcommand{\inter}{\mathrm{int}}

\newcommand{\mM}{\mathcal{M}}

\newcommand{\R}{\mathbb R}
\newcommand{\Z}{\mathbb Z}
\newcommand{\Q}{\mathbb Q}
\newcommand{\CP}{\mathbb{CP}}

\newcommand{\Emb}{\textup{Emb}}
\newcommand{\Mnfld}{\textsf{Mfd}}
\newcommand{\Mnfldbd}{\textup{Mfd}^\partial}

\newcommand{\tors}{\textup{tors}}

\newcommand{\colim}{\textup{colim}\,}
\newcommand{\hocolim}{\textup{hocolim}\,}

\newcommand{\mc}{\mathcal}
\newcommand{\op}{\textup{op}}

\newcommand{\Diff}{\textup{Diff}\,}

\newcommand{\Lr}{L^R}

\newcommand{\ChPerf}{\textup{Ch}^{\textup{perf}}}

\newcommand{\Xm}{X_m}

\newcommand{\sd}{\mathrm{sd}}

\newcommand{\dotp}{{{\raisebox{.2ex}{\scalebox{.5}{$\bullet$}}}}}

\DeclareMathOperator{\Cat}{\mathrm{Cat}}

\usepackage{xcolor}
\definecolor{seagreen}{RGB}{46,139,87}
\definecolor{maroon}{RGB}{128,0,0}
\definecolor{darkviolet}{RGB}{148,0,211}

\newcommand{\sbt}{\,\begin{picture}(-1,1)(0.5,-1)\circle*{1.8}\end{picture}\hspace{.05cm}}

\newlength{\storeparskip}
\begin{document}
\title[Parametrized scissors congruence $K$-theory of manifolds]{Parametrized scissors congruence $K$-theory of manifolds and cobordism categories}

\author[Mona Merling]{Mona Merling}
\address{Department of Mathematics, The University of Pennsylvania}
\email{mmerling@math.upenn.edu}

\author[George Raptis]{George Raptis}
\address{
Department of Mathematics, Aristotle University of Thessaloniki, 54124 Thessaloniki, Greece}
\email{raptisg@math.auth.gr}

\author[Julia Semikina]{Julia Semikina}
\address{Paul Painlevé Mathematics Laboratory, University of Lille}
\email{iuliia.semikina@univ-lille.fr}

\begin{abstract}
We introduce a parametrized version of scissors congruence $K$-theory of manifolds with tangential structure, which includes a topologized version of the scissors congruence $K$-theory of oriented manifolds as a special case.  We examine the relation of this $K$-theory spectrum with cut-and-paste invariants, the (parametrized) cobordism category and with (bivariant) algebraic $K$-theory of spaces. We show that the scissors congruence $K$-theory of oriented manifolds agrees on $\pi_0$ with a version of the oriented cobordism category where we allow cobordisms to have free boundaries. Lastly, we show that the spectrum level refinement of the Euler characteristic from the scissors congruence $K$-theory to $K(\Z)$ detects on $\pi_1$ the Kervaire semicharacteristic. \end{abstract}

\maketitle

\begingroup%
\setlength{\parskip}{\storeparskip}
\setcounter{tocdepth}{1}
\tableofcontents
\endgroup%

\section{Introduction}

Generalizing Segal's category of conformal surfaces \cite{Segal_conformal}, Galatius, Madsen, Tillmann, and Weiss introduced a topological category $\Cob_d$ of smooth oriented cobordisms in \cite{GMTW}, 
whose space of objects is equivalent to $\underset{M}{\bigsqcup} B\Diff(M)$, where $M$ ranges over oriented closed $(d-1)$-dimensional manifolds, and whose space of (non-identity) morphisms is equivalent to $\underset{W}{\bigsqcup} B\Diff_\partial(W)$, where $W$ ranges over diffeomorphism classes of oriented compact $d$-dimensional cobordisms. The homotopy type of the classifying space $B\Cob_d$ was identified in \cite{GMTW} with the infinite loop space of the Thom spectrum $MTSO(d)$, via a parametrized version of the Pontryagin-Thom construction. This was notably used to give an alternative proof of the generalized Mumford conjecture describing the stable cohomology of the moduli spaces of complex curves and has led since then to many other groundbreaking generalizations, novel applications, and new computations concerning the homotopy types of (stabilized) diffeomorphism groups. 

The infinite loop space $B\Cob_2$ is, at least after $p$-adic completion, closely related to the topological cyclic homology of the sphere spectrum, a relationship that Madsen emphasized in his 2006 ICM talk and which remains mysterious. Pursuing this connection, B\"okstedt and Madsen  \cite{bokstedt_madsen} constructed an infinite loop map 
$$\tau\colon \Omega B\Cob_d\to A(BSO(d)),$$ which roughly views a composite of $n$ cobordisms as an $n$-step filtration of the composite manifold, and records the orientation by its classifying map to $BSO(d).$ Raptis and Steimle showed in \cite{RS14, RS17} that the restriction of $\tau$ to $B\Diff(M)$ for a closed $d$-dimensional manifold $M$ agrees with the parametrized $A$-theory Euler characteristic of the universal bundle over $B\Diff(M)$, and moreover, the map $\tau$ factors up to homotopy through the unit map $Q(BSO(d)_+) \to A(BSO(d))$. More recently, Steinebrunner \cite{steinebrunner2021classifying} gave a compelling interpretation of $TC(\Omega X)$ for simply connected $X$ of finite type in terms of reduced 1-dimensional cobordisms in $X$.

This paper goes  back to the idea of finding a $K$-theoretic approximation of the cobordism category, but considering scissors congruence $K$-theory of manifolds instead of $A$-theory.

In the 70s, Karras, Kreck, Neumann, and Ossa \cite{KKNO} introduced scissors congruence groups for closed manifolds, also known as $\SK$-groups (“schneiden und kleben”, German for “cut and paste”). The $SK$-groups were generalized in \cite{WIT} to the setting of oriented manifolds with boundary and they were extended to a $K$-theory spectrum using the framework of $K$-theory for squares categories \cite{CKMZ}. As an application of this refinement, the authors constructed a map of spectra to $K(\Z)$ which recovers the Euler characteristic.

In this paper, we construct a map from the cobordism category to (a topologized version of) the scissors congruence $K$-theory $K^{\square}(\Mnfld^{\partial, d}_{\Delta})$ of oriented manifolds with boundary. 
We show that the scissors congruence $K$-theory of oriented manifolds mediates between the oriented cobordism category and $A$-theory. More precisely, we construct a factorization 
\[\xymatrix{
& K^{\square}(\Mnfld^{\partial, d}_{\Delta}) \ar[dr] & \\
\Omega B\Cob_d \ar[ur] \ar[rr]^\tau && A(BSO(d)).
}\]

 In addition, following the bivariant perspective of \cite{RS17, RS20}, we generalize scissors congruence $K$-theory (and the factorization above) to the parametrized setting of manifold bundles with tangential structure. Specifically, we introduce squares categories of manifold bundles (over a fixed base) with $\theta$-structure for any parametrized tangential structure $\theta$, and demonstrate that these formally fit into the setting of bivariant theories in the sense of \cite{RS17, RS20}. We recall that in \cite{RS17, RS20}, the authors introduced parametrized cobordism categories and a bivariant extension of the B{\"o}kstedt-Madsen map, and they proved its factorization through the unit map as maps of bivariant theories -- this can be seen as a strong form of the Dwyer-Weiss-Williams index theorem \cite{DWW}.  Furthermore, we also examine the relation of parametrized scissors congruence $K$-theory to cut-and-paste invariants for manifold bundles.

 The map from the cobordism category $\Cob_d$ to scissors congruence $K$-theory of oriented $d$-manifolds with boundary is actually defined on the larger category $\Cobf_d$ of cobordisms that are allowed to have ``free" boundary components (in addition to the ones that represent the source and the target of the morphism). In fact, this is a full subcategory of the cobordism category $\Cob^\partial_d$ of manifolds with boundary \cite{Genauer12}, restricted to the objects which are closed oriented $(d-1)$-manifolds. The map $\Omega B\Cob_d \to K^{\square}(\Mnfld^{\partial, d}_{\Delta})$ is clearly not an equivalence -- for example, the spaces have different $\pi_0$. On the other hand, $\Omega B\Cobf_d$ is a better approximation to scissors congruence $K$-theory of oriented manifolds, and it suggests the following question which remains open:

\noindent \textbf{Question.}\emph{ Is the map $\Omega B\Cobf_d\to K^{\square}(\Mnfld^{\partial, d}_{\Delta})$ (or its versions for different parametrized tangential structures $\theta$) a weak homotopy equivalence?}

As a positive step in this direction, we calculate $\pi_0(\Omega B\Cobf_d)$ and show that it is isomorphic to the scissors congruence group $SK_d^\partial$ of oriented manifolds with boundary. This is used to conclude that the map above is a $\pi_0$-isomorphism. 

Lastly, concerning $\pi_1$, we construct elements in $K^{\square}_1$ of the scissors congruence $K$-theory of (unoriented) manifolds, which arise from diffeomorphisms of closed smooth manifolds. We compute their images under the Euler characteristic map to $K(\Z)$ and conclude that the Euler characteristic map detects the Kervaire semicharacteristic on $\pi_1$. The Kervaire semicharacteristic is an $SKK$-invariant; this is a finer relation than $SK$-invariance, where the equivalence relation keeps track of the glueing diffeomorphisms in the cut-and-paste operation \cite{KKNO}. The associated groups $SKK_d$ have a cobordism interpretation and arise also as fundamental groups of the cobordism category \cite{KKNO, ebert2013vanishing, bokstedt2014geometric}.

\subsection{Related work} We point out some related work in the literature on the connections between algebraic $K$-theory and cobordism categories. In \cite{RS19}, a cobordism model $\Cob(\mc C)$ is introduced for the $S_\dotp$-construction of a Waldhausen category $\mc C$, so that $\Omega B\Cob(\mc C)\simeq K(\mc C)$, and the authors revisited the B{\"o}kstedt-Madsen from this viewpoint. Another $K$-theory of manifolds was also considered in \cite{HRS24} to study a different relation between cut-and-paste invariants and cobordism theory, namely, to lift the cobordism cut-and-paste groups $\overline{SK}$ to a spectrum. More recently, in \cite{CalleSarazola}, an $S_\dotp$-construction for squares categories is also introduced. In a different direction, in \cite{merling2025scissors}, a scissors congruence $K$-theory spectrum for equivariant manifolds is constructed (which can also be promoted to a topologized version using the methods of the present paper); an interesting future research direction would be to explore its relationship with equivariant cobordisms.

\subsection{Outline}
In \autoref{parametrizedK} we recall the set-up of bivariant theories from \cite{RS17, RS20} and we define the bivariant theory of parametrized manifolds with tangential $\theta$-structure, as valued in the category of squares categories. Using this, we define the bivariant theories of squares and scissors congruence $K$-theory of parametrized $\theta$-manifolds. In \autoref{SKsection}, after reviewing the classical definitions of scissors congruence groups of manifolds, also known as $SK$-groups, we introduce $SK$-groups for parametrized $\theta$-manifolds and show that these arise as $\pi_0$ of the scissors congruence $K$-theory of parametrized $\theta$-manifolds. In \autoref{section: map from cob} we construct the map from the cobordism category (where cobordisms are alllowed to have free boundary) to scissors congruence $K$-theory. In \autoref{sec:Atheory} we construct the map from scissors congruence $K$-theory of manifolds to $A$-theory and show that this factorizes the B{\"o}kstedt-Madsen map through scissors congruence $K$-theory. In \autoref{sec: pi_0 computation} we show that the comparison map from the oriented cobordism category (with free boundary components) to scissors congruence $K$-theory of oriented manifolds is a $\pi_0$-isomorphism. In \autoref{section: K_1}, we construct elements in (unoriented) $K^{\square}_1$ from diffeomorphisms, and show that these recover the Kervaire semicharacteristic of the manifold under the Euler characteristic map to $K(\Z)$. Lastly, in an appendix, we show that we can also construct elements in $K_1$ for the (non-topologized) scissors congruence $K$-theory from \cite{WIT} for unoriented manifolds, so that the same computation recovering the Kervaire semicharacteristic also holds for that construction.

\subsection{Acknowledgments} 
The authors would like to thank the following people for helpful discussions: Maxine Calle, Johannes Ebert, Søren Galatius,
Renee S. Hoekzema, Achim Krause, Daniel Kasprowski, Sander Kupers, Wolfgang L{\"u}ck, Cary Malkiewich, Thomas Nikolaus, Oscar Randal-Williams, Carmen Rovi, Wolfgang Steimle, Antoine Touz\'e, and Laura Wells.  M.M. was partially supported by NSF DMS grants CAREER 1943925 and FRG 2052988. G.R. was supported by the Hellenic Foundation for Research and Innovation (H.F.R.I.) under the ``3rd Call for H.F.R.I.’s Research Projects to Support Faculty Members \& Researchers'' (Project Number: 25480) and partially supported by the SFB~1085 - \emph{Higher Invariants} funded by the DFG.  J.S. was partially supported by the Labex CEMPI (ANR-11-LABX-0007-01).

\section{Scissors congruence $K$-theory of manifolds}\label{parametrizedK}

\subsection{Bivariant theories}\label{subsec: bivariant} In this section, we follow \cite{RS17, RS20} to recall the set-up of bivariant theories, which was inspired by \cite{FM81}. 

\begin{defn}\label{thetastr}
    
    Let us fix an integer $d \geq 0$. A family of vector bundles of rank $d$ is given by a triple $\theta=( B, ~ p \colon X \to B, ~\xi \colon V \to X ),$ where $B$ is a space with the homotopy type of a CW-complex, $p$ is a fibration and $\xi$ is a numerable (real) vector bundle of rank $d$.
\end{defn}

Given such a family $\theta = (B, p, \xi)$ and a map $g \colon B' \to B$, where $B'$ has the homotopy type of a CW-complex, we define a new family $g^*\theta = (B', g^*p, g^*\xi)$ defined by pullback. (We assume for convenience that both $X$ and $V$ are fiberwise subsets over $B$ of a fixed ambient set, so that pullbacks are strictly functorial, see \cite{RS20}.) As in \cite[Section 2.1]{RS20},  there is a category $\Biv$, whose objects are the triples $\theta$, that encodes the bivariance of $\theta$: there is contravariant functoriality $\theta \to g^* \theta$ with respect to maps of base spaces $g \colon B'\to B$  (\emph{contravariant operations}), and covariant functoriality with respect to (parametrized) bundle maps $b \colon \theta \to \theta'$ that are given by a fiberwise map $X \to X'$ over the same base space $B$ together with a fiberwise isomorphism of vector bundles $\xi\to \xi'$ (\emph{covariant operations}), i.e., a pullback square over $B$:
$$
\xymatrix{
V \ar[d]_{\xi} \ar[r] & V' \ar[d]^{\xi'} \\
X \ar[r] & X'. 
}
$$
We will refer to such morphisms as (parametrized) bundle maps. 

\begin{defn} \label{def: bivariant theory}
A \textit{bivariant theory} with values in a category $E$ is a functor $
F \colon \Biv \to E.$
\end{defn}
\noindent Unraveling the definition, a bivariant theory $F$ consists of the following data:
\begin{itemize}
    \item an object $F(\theta)$ of $E$ for each family of vector bundles $\theta$ of rank $d$; 
    \item for each map $g \colon B' \to B$, a morphism (contravariant operation) $g^* \colon F(\theta) \to F(g^* \theta)$ in $E$ which is natural in $g$;
    \item for each bundle map $b \colon \theta \to \theta '$ (parametrized over the same base space $B$) a morphism (covariant operation) $b_* \colon F(\theta) \to F(\theta')$ in $E$ which is natural in $b$. 
\end{itemize}
In addition, the covariant and contravariant operations commute with each other in the sense that each of the following squares is commutative:
  \begin{center}
         \begin{tikzcd}
            F(\theta) \arrow[d, swap]{}{b_*} \arrow[r]{}{g^*} & F(g^* \theta) \arrow[d]{}{(g^*b)_*}\\
            F(\theta') \arrow[r]{}{g^*} & F(g^* \theta')
         \end{tikzcd}
  \end{center}
where $g^*b$ denotes the induced bundle map by pullback along $g$. This notion of bivariant theory was used in \cite{RS17, RS20} in the context of parametrized cobordism categories with (parametrized) tangential structures. 

Following \cite{RS17, RS20}, we consider for any $\theta = (B, p \colon X \to B, \xi \colon V \to X)$ and $n \geq 0$ the following object in $\Biv$
$$\theta_n = (B \times \Delta^n, \ p_n  \colon X \times \Delta^n \xrightarrow{p \times \id} B \times \Delta^n, \ \xi_n  \colon V \times \Delta^n \xrightarrow{\xi \times \id} X \times \Delta^n)$$
where $\Delta^n$ denotes the standard (topological) $n$-simplex. These objects determine a simplicial object 
$$\theta_{\sbt} \colon \Delta^{\op} \to \Biv, \ [n] \mapsto \theta_n,$$
where the face and degeneracy maps are defined using the contravariant functoriality with respect to the base. Thus, we may precompose any bivariant theory $F \colon \Biv \to E$ with $\theta_{\sbt}$ and obtain a simplicial object in $E$ for each $\theta \in \Biv$ (or a simplicial bivariant theory). As shown in \cite{RS17, RS20} (for space- or spectrum-valued bivariant theories), this operation takes into account the homotopy theory inherent in $\Biv$ and, after geometric realization, turns a bivariant theory into a homotopy invariant one. 

\subsection{Parametrized manifolds}

We introduce parametrized tangential structures and in the next section we generalize the definition of the squares category of $d$-dimensional manifolds with boundary from \cite{WIT} to the setting of parametrized manifolds with tangential structure. 

Each triple $\theta$ in \autoref{thetastr} gives rise to a notion of tangential structure for $B$-parametrized families of smooth $d$-manifolds with boundary. Recall that the \textit{vertical tangent bundle} of a smooth manifold bundle $\pi\colon E\to B$ with $d$-dimensional fiber $M$ is defined as the $d$-dimensional vector bundle $T_{\pi}E:=P\times_{\Diff(M)} TM \to E $, where $P$ is the principal $\Diff(M)$-bundle associated to $\pi$.

\begin{defn}
  Let $\theta=(B, p, \xi)$ be as in \autoref{thetastr}. A \emph{$\theta$-manifold} (with boundary) consists of
  \begin{enumerate}
      \item a numerable fiber bundle $\pi \colon E \to B$ with fibers given by compact smooth $d$-manifolds (possibly with boundary); 
      \item a bundle map
       $$\begin{tikzcd}
    	T_{\pi} E\arrow[r] \arrow[d,"\tau_{\pi} E"'] & V\arrow{d}{\xi}\\
    	E\arrow{r}{f} & X
    	\end{tikzcd} $$
     from the vertical tangent bundle $\tau_{\pi}E$ of $\pi \colon E \to B$  to $\xi$ such that $\pi = p f$ (i.e., $f$ is fiberwise over $B$). 
  \end{enumerate} 
     A map of $\theta$-manifolds is a map of manifold bundles over $B$, which is also a map over $X$. 
\end{defn}

We will also refer to $\theta$-manifolds as parametrized manifolds or manifold bundles with $\theta$-structure. Note that a $\theta$-structure on a manifold bundle $\pi \colon E \to B$ is essentially given by a lift $f \colon E \to X$ of $\pi$ along $p$ together with an isomorphism of vector bundles $\tau_{\pi} E \cong f^*\xi$. In the special case where $B=\ast$, this is just given by a compact smooth $d$-dimensional manifold $M$ with a structure map $f \colon M\to X$ together with an isomorphism $TM \cong f^*\xi$. We also have the following special cases which will be of interest later.

\begin{exmp}\label{orientedex}
   Let $\theta=(*,\  BSO(d)\times X, \ \gamma_d \times \id_X \colon E(\gamma_d) \times X \to BSO(d)\times X)$, where $\gamma_d$ denotes the universal oriented vector bundle of rank $d$. Then a $\theta$-manifold is an oriented singular $d$-manifold (possibly with boundary) in $X$ in the standard sense of cobordism theory. Similarly, replacing $BSO(d)$ with $BO(d)$ in the definition of $\theta$ gives an unoriented singular $d$-manifold in $X$. In particular, when $X=\ast$, this is simply a compact oriented, resp. unoriented, smooth $d$-manifold. 
\end{exmp}

\subsection{The squares $K$-theory of $\theta$-manifolds}\label{thetamanifolds} 
A formalism for $K$-theory of squares categories was introduced in \cite{CKMZ}. It was previewed in \cite{WIT} where the example of the squares category of $d$-dimensional manifolds with boundary was introduced. Loosely, a squares category generalizes both categories with exact sequences and those with subtractive sequences -- it is defined by a structure of distinguished squares. $K$-theory of squares categories aims to break up squares instead of exact or subtractive sequences, and it can be applied even in contexts which lack general pushouts or cofiber sequences. From a different viewpoint, a closely related notion was introduced earlier in the context of $K$-theory for triangulated categories \cite{Neeman} and has also been considered for $K$-theory in the context of higher homotopy categories \cite{Raptis22}.

Next we introduce the squares category of $\theta$-manifolds. We refer the reader to \cite[Definition 1.4]{CKMZ} for the definition of squares categories and functors between them. 

We recall the  notion of $SK$-embedding of manifolds with boundary from \cite{WIT}, which we then extend to parametrized manifolds.

\begin{defn}\label{SKemb}
Let $M$ and $N$ be compact smooth $d$-manifolds (possibly with boundary). An \emph{$SK$-embedding} from $N$ to $M$ is a smooth embedding $f\colon N\to M$ such that each connected component of $\partial N$ is either mapped entirely onto a boundary component of $M$ or entirely into the interior of $M$. 
\end{defn}

For any  manifold bundle $\pi \colon E \to B$ with smooth compact $d$-dimensional fibers, we denote by $\pi_{\partial} \colon \partial E \to B$ the associated $(d-1)$-dimensional manifold bundle obtained from $E$ by restricting to the boundary of the manifold in each fiber.  
Given a map of $\theta$-manifolds $f \colon E \to E'$, it is not enough to define the $SK$-embedding condition from \autoref{SKemb} simply fiberwise, since this is not strong enough to allow for a well-behaved notion of complement of $E$ in $E'$, which we will need in order to talk about cut-and-paste decompositions of parametrized manifolds. 

\begin{exmp}\label{jumpsex}
    Consider the map of trivial bundles $E_1, E_2$ over the base space $[0,1]$ with fibers given by $[0,1]$ and $[0,2]$ depicted in \autoref{fig:badSKexample}.
    \begin{figure}
        \centering
        \includegraphics[width=0.5\linewidth]{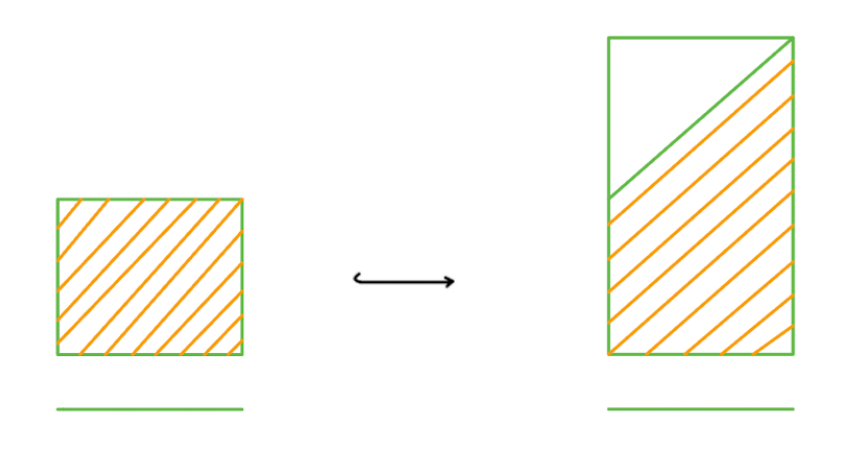}
        \caption{A fiberwise $SK$-embedding which is not an $SK$-embedding of bundles}
        \label{fig:badSKexample}
    \end{figure}
    In this example, the boundary of each fiber of $E_1$ is two points, one mapping into the interior and one mapping to the boundary of $E_2$ for each $b
    \in [0,1)$, and the boundary of the fiber above $b=1$ consists of two points that both map to the boundary of $E_2$. 
\end{exmp}

\begin{defn}\label{SKembbundle}
   A map $f \colon E \to E'$ of $\theta$-manifolds is called an $SK$-embedding of $\theta$-manifolds if it is an $SK$-embedding fiberwise over $B$ and $ f^{-1}(\partial E') \subseteq \partial E$ is open. 
\end{defn}

The extra condition in \autoref{SKembbundle} ensures that we avoid  ``jumps'' of the boundary components of the fibers mapping into the interior such as the ones in \autoref{jumpsex}, and that we can define a complement bundle of an $SK$-embedding of $\theta$-manifolds. Its main function is essentially that it allows for the existence of a parametrized tubular neighborhood. 

Let $f \colon E \to E'$ be an $SK$-embedding of $\theta$-manifolds. We make the following definitions of the concepts of interior and complement; the latter makes sense in view of our definition of $SK$-embedding of manifold bundles.

\begin{defn}\label{def: intcomp}
   The \emph{interior bundle} of a manifold bundle $E$, $\inter(E)$, is the manifold bundle given by the complement $E-\partial E$ of $\partial E$.   
   For an $SK$-embedding of manifold bundles $f\colon E\to E'$, the \emph{complement bundle} $E^c$ is a subbundle of $E'$ given by the closure of the complement $E'-f(E)$ of $f(E)$ in $E'$. \end{defn}

It is straightforward to check that $\inter(E)$ is indeed a manifold bundle with noncompact fiber, and $E^c$ is a manifold bundle with fiber given by the closure of $M'-f(M)$ in $M'$, where $M, M'$ are the corresponding fibers of $E, E'$. This fiber is indeed a manifold since $f$ is a fiberwise $SK$-embedding. In addition, given the existence of 
a tubular neighborhood for the fiberwise embedding $f$, $f$ can be identified, locally in $B$ and up to fiberwise diffeomorphisms, with a constant $SK$-embedding (see \cite[Chapter 11, esp. Proposition 11.20]{Crabb-James}). The complement $E^c$ has  a $\theta$-structure induced by the $\theta$-structure of $E'$ by restriction.

\begin{defn}
    Let $\theta=(B, p, \xi)$ be as in \autoref{thetastr}. We define $\Mtheta$ to be the category with squares where 
    \begin{itemize}
    \item the objects are $\theta$-manifolds (with boundary);
    \item the horizontal and vertical maps are $SK$-embeddings $f \colon E \hookrightarrow E'$  of $\theta$-manifolds; 
    \item the distinguished squares 
    \begin{center}
    		\begin{tikzcd}[column sep=tiny, row sep=tiny]
    			E_1\arrow[rr, hook] \arrow[dd, hook] && E_2\arrow[dd, hook]\\
    			&\square &\\
    			E_3\arrow[rr, hook] && E_4
    		\end{tikzcd}
    	\end{center}
        are commutative squares of $SK$-embeddings of $\theta$-manifolds which are (fiberwise) pushouts in the category of spaces. 
    \end{itemize}
\end{defn}
 
It is immediate from the definition that $\Mtheta{}$ satisfies the axioms for a category with squares: squares compose vertically and horizontally, a commutative square with parallel identities, either vertically or horizontally, is distinguished, and lastly, the empty manifold (bundle) is initial.

The category $\Mtheta$ has a symmetric monoidal structure in the sense of \cite[Definition 1.13]{CKMZ} given by the disjoint union $\sqcup$. In particular, the isomorphism classes of objects in $\Mtheta{}$ form a commutative monoid.

We recall the construction of the $K$-theory of squares categories \cite[\S 2]{CKMZ} in our example. 
Let $T_n \Mtheta$ be the category with objects given by length $n$ sequences of composable $SK$-embeddings of $\theta$-manifolds  
$$E_1\hookrightarrow E_2\hookrightarrow \dots \hookrightarrow E_n,$$
and the morphisms are given by commutative diagrams 
$$\begin{tikzcd}[column sep=tiny, row sep=tiny]
    E_1 \ar[rr, hook] \ar[dd, hook] && E_2 \ar[rr, hook] \ar[dd, hook] &&\dots \ar[rr, hook] && E_n \ar[dd, hook]\\
    &\square &&\square && \square &\\
    E'_1 \ar[rr, hook] && E'_2 \ar[rr, hook] &&\dots  \ar[rr, hook] && E'_n,
\end{tikzcd}$$
where all squares are distinguished. These categories assemble into a
  simplicial category which we denote $T_\dotp \Mtheta$. Then the squares $K$-theory of $\Mtheta$ is defined by 
  \[K^\square(\Mtheta) = \Omega_{\varnothing} |N_\dotp T_\dotp\Mtheta|.\]
  Here, $\Omega_{\varnothing}$ is the loop space based at the empty manifold $\varnothing \in N_0
  T_0 \Mtheta$, which is an initial object in both the vertical and horizontal directions.  

  \begin{notation}\label{not:squares}
      For a squares category $\C$ we will denote by $N_\dotp^\square \C$ the diagonal of the bisimplicial set $N_\dotp T_\dotp \C$, i.e., the simplicial set whose $n$ simplices are $n\times n$ grids of distinguished squares. It is well known that its geometric realization agrees with the geometric realization of the bisimplicial set. 
  \end{notation}
  
  By \cite[Proposition 2.5]{CKMZ}, the squares $K$-theory space $K^\square(\Mtheta)$ is an infinite loop space. By abuse of notation, from now on we will also use the same notation for the associated $\Omega$-spectrum where needed.

\subsection{Bivariant $K$-theory of manifolds}

 We show that the categories with squares $\Mtheta$ assemble into a bivariant theory in the sense of  \autoref{def: bivariant theory}, that is, we construct a functor 
$$\Mnfld^{(-)}\colon \Biv \to \Cat^\square,$$ 
where $\Cat^\square$ denotes the category of squares categories and functors of squares categories.

\noindent  Explicitly, as in \cite[Section 2.1]{RS20}, we specify the following data: 

(a) On objects, we assign to any $\theta=(B,p \colon X \to B, \xi \colon V \to X)$ the category with squares $\Mtheta$.

(b) Given a map $g \colon B'\to B$, we define a functor (contravariant operation) of categories with squares  $g^* \colon \Mnfld^{\theta} \to \Mnfld^{g^* \theta}$. An object $(\pi \colon E\to B, \tau_{\pi}E \to \xi)$ is sent to the object $(g^*\pi \colon g^* E \to B', \tau_{g^* \pi}g^*E \to g^*\xi)$ defined by taking pullbacks
$$\begin{tikzcd}
&  T_{\pi} E \arrow[dr] \arrow[d] & \\
g^*E \arrow[r] \arrow[ddr, "g^*\pi", labels=below left] \arrow[dr] & E\arrow[dr] \arrow[ddr]  & V \arrow[d, "\xi"]  \\
&    	g^* X \arrow[r, crossing over] \arrow[d] & X\arrow[d, "p"]  \\
&            B' \arrow[r, "g"] & B.
    	\end{tikzcd} $$
It is straightforward to see how this map extends to morphisms similarly by pullback. It is also evident that this assignment preserves $SK$-embeddings, distinguished squares and the empty manifold, hence $g^* \colon \Mtheta \to \Mnfld^{g^* \theta}$ is indeed a functor of categories with squares.

(c) Given a bundle map $b \colon \theta \to \theta'$, we construct a morphism (covariant operation) $b_* \colon \Mtheta \to \Mnfld^{\theta'}$. An object $(\pi \colon E \to B, \tau_{\pi} E \to \xi)$ is sent to the same manifold bundle $\pi \colon E\to B$ equipped with the induced $\theta'$-structure that is given by composition with the bundle map $b$:
$$\begin{tikzcd}
    	T_{\pi}E \arrow[r] \arrow[d] & V \arrow{d}{\xi} \arrow[r] & V' \arrow{d}{\xi'} \\
    	E \arrow[rd] \arrow[r] & X \arrow{d}{p} \arrow[r] & X' \arrow{ld}{p'}\\
            & B. &
    	\end{tikzcd} $$
It is again straightforward to extend the definition to morphisms and to see that it preserves the initial object, $SK$-embeddings and pushout/distinguished squares. 

Both contravariant and covariant functorialities in (b) and (c) are natural and commute with each other, so we obtain a functor (bivariant theory) $\Mnfld^{(-)}$, as desired. 

We consider the associated \emph{simplicial thickening} of the bivariant theory $\Mnfld^{(-)}$ (cf.  \autoref{subsec: bivariant}) --a bivariant theory with values in simplicial objects-- $$\Mnfld^{(-)_{\sbt}} \colon \Biv\to \mathrm{Fun}(\Delta^{\op}, \Cat^{\square}), \ \theta \mapsto ([n] \mapsto \Mnfld^{\theta_n}).$$

\begin{defn}
 The scissors congruence $K$-theory of $\theta$-manifolds with boundary $K^\square(\Mtheta_\Delta)$ is the geometric realization of the simplicial space $[n]\mapsto K^\square(\Mnfld^{\theta_n})$.
\end{defn}

Note that since the spaces $|N_\dotp^\square \Mnfld^{\theta_n}|$ are path-connected, geometric realization commutes up to homotopy equivalence with taking loop spaces, so the realization of this simplicial space agrees canonically up to homotopy equivalence with $\Omega |[n]\mapsto |N_\dotp^\square \Mnfld^{\theta_n}||$. As shown in \cite[Section 2]{RS17}, the simplicial thickening of the squares $K$-theory of $\Mnfld^{(-)}$ is a bivariant theory which is homotopy invariant in both variances.

\begin{rem}
$K^\square(\Mtheta)$ and $K^\square(\Mtheta_{\Delta})$ have very different homotopy types in general, because the former does not take into account the topology of the spaces of $\theta$-structures. For example, for $d=0$ and $\theta = (\ast, X, \id_X \colon X \to X)$, we have $K^\square(\Mtheta) \simeq Q(X^{\delta}_+)$ ($X^{\delta}$ denotes the underlying set of $X$ equipped with the discrete topology), whereas $K^\square(\Mtheta_{\Delta}) \simeq Q(X_+)$.    
\end{rem}

\section{Comparison with cut-and-paste invariants}\label{SKsection}

\subsection{Recollections on classical cut-and-paste groups of manifolds} 

The classical scissors congruence group $SK_d$ of smooth closed  oriented $d$-dimensional manifolds was first introduced in \cite{KKNO} as the Grothendieck group of the monoid $\mathcal{M}_d$ of diffeomorphism classes of smooth closed oriented $d$-dimensional manifolds $[M]$ under disjoint union, modulo a ``scissors congruence" or ``cut-and-paste" relation ($SK$ stands for ``schneiden und kleben"=``cut and paste" in German). The $SK$-relation is generated by the following cut-and-paste operation: cut a $d$-dimensional manifold $M$ along a codimension 1 smooth submanifold  $\Sigma$ with trivial normal bundle that separates $M$ (i.e., the complement of $\Sigma$ in $M$ is a disjoint union of two components $M_1$ and $M_2$\footnote{We do not lose any generality by asking for the cut to be separating. If it is not, we may instead make two cuts in a tubular neighborhood of the submanifold we are cutting along.}), each with (oriented) boundary diffeomorphic to $\Sigma$. Then paste back the two pieces together along an orientation-reversing diffeomorphism $\phi\colon \partial M_1 \to \partial M_2.$ We say $M$ and $M_1\cup_\phi M_2$ are ``cut-and-paste equivalent" or ``scissors congruent" (by a single cut-and-paste operation in this case).

\begin{figure}[h!]
\includegraphics[scale=.3]{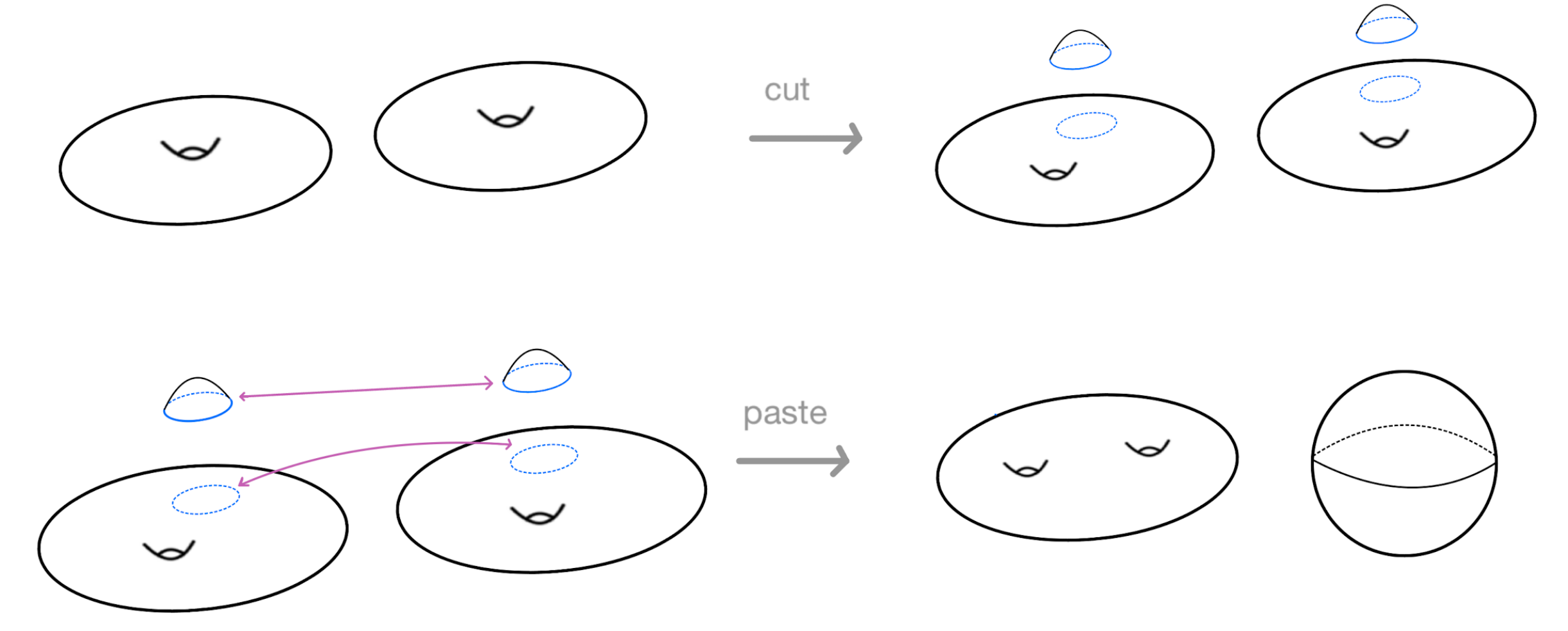}
\caption{Example of a cut-and-paste operation illustrating that $T\sqcup T$ is scissors congruent to $T \# T \sqcup S^2$ (here $T$ denotes the $2$-torus). }
\centering
\end{figure}

A version  of these groups for manifolds with boundary was introduced in \cite{WIT} (this differs from the one in \cite{KKNO}). The cut-and-paste operation is defined completely analogously to the one for closed manifolds, but insisting that we do not allow boundaries to be cut, and we require that  all boundaries which come from cutting to be pasted back together, leaving the existing boundaries of a manifold untouched by the cut-and-paste operation. We recall the definition of the $SK$-group for manifolds with boundary.

\begin{defn}[\cite{WIT}]
   The cut-and-paste group for oriented $d$-dimensional manifolds with boundary, $\SK^{\partial}_d$, is the quotient of the free abelian group on the diffeomorphism classes of smooth compact oriented $d$-dimensional manifolds (possibly with boundary) by the subgroup generated by the following relations:
  \begin{enumerate}
    \item $[M \sqcup N] \sim [M] + [N]$;
     \item Given smooth compact oriented $d$-manifolds $M_1, M_2$, closed submanifolds $\Sigma \subseteq \partial M_1$ and $\Sigma' \subseteq \partial M_2$, and orientation-preserving diffeomorphisms $\phi, \psi \colon \Sigma \to \Sigma'$, then
\[ [M_1 \cup_{\phi} \overline{M}_2] \sim [M_1 \cup_{\psi} \overline{M}_2],\]
where $\overline{M}$ denotes the manifold $M$ with the reverse orientation.
   \end{enumerate}
\end{defn}

It was proved in \cite[Theorem 2.10]{WIT} that the group $SK_d^\partial$ fits into a split exact sequence with kernel the classical group $SK_d$ for closed manifolds and cokernel the Grothendieck group of the commutative monoid of diffeomorphism classes of smooth closed oriented $(d-1)$-dimensional and nullbordant manifolds, under disjoint union.

\subsection{The cut-and-paste groups $SK(\theta)$}
We extend the definition of the cut-and-paste groups to manifold bundles with $\theta$-structure.   

Recall that $\partial E$ denotes $(d-1)$-dimensional manifold bundle (over $B$) obtained from a bundle $\pi \colon E \to B$ of smooth compact $d$-dimensional manifolds, by restricting to the boundary of the manifold 
in each fiber.

\begin{defn} \label{def: cut-and-paste of theta mnfds}
    Let $\theta = (B, p, \xi)$ be as in  \autoref{thetastr}, and let $E$ and $E'$ be $\theta$-manifolds. We say that $E$ and $E'$ are \emph{related by a cut-and-paste operation} if there exist $\theta$-manifolds $E_1, E_2, E'_1, E'_2$  and $SK$-embeddings of $\theta$-manifolds $E_i \subseteq E$ and $E'_i\subseteq E'$ for $i=1,2$ such that 
    $$E = E_1 \cup_{E_0} E_2 \ \text{ and } E' = E'_1 \cup_{E'_0} E'_2,$$
    where $E_0 \subseteq \partial E_i$, $E'_0 \subseteq \partial E'_i$ are subbundles of closed smooth manifolds for $i=1,2$, and  there are diffeomorphisms of $\theta$-manifolds 
 $$E_1 \cong E'_1 \ \text{ and } E_2 \cong E'_2$$
 which send $E_0$ to $E'_0$.  
\end{defn}

Classically, the definition of the cut-and-paste operation for oriented manifolds is stated in terms of glueing along orientation-reversing diffeomorphisms of a codimension 1 submanifold that we cut along. Of course, this could be equivalently restated in terms of decomposing manifolds. In the parametrized setting, we state the definition in terms of decomposing manifolds to avoid talking about glueing manifolds with $\theta$-structure.

Let $\mathcal{M}^{\theta}$ denote the abelian monoid of isomorphism classes of $\theta$-manifolds under disjoint union. The $\SK$-equivalence relation $\sim_{SK}$ on $\theta$-manifolds is the equivalence relation on $\mathcal{M}^{\theta}$ that is generated by cut-and-paste operations. Note that the $SK$-equivalence relation for $\theta$-manifolds does not take into account homotopies of (parametrized) tangential structures. For example, note that even in the parametrized oriented case corresponding to \autoref{orientedex}, a $\theta$-manifold is not just an oriented manifold, but it comes equipped with a structure map to $BSO(d)$ as opposed to just the homotopy class of such a map, i.e., an orientation. In order to obtain the cut-and-paste group for $\theta$-manifolds, we thus consider a further quotient. 

 \begin{defn}\label{SKtheta}
    Let $\theta = (B, p, \xi)$ be as in  \autoref{thetastr}. The \emph{cut-and-paste group} of $\theta$-manifolds $SK(\theta)$ is defined as the quotient group
     $$SK(\theta)= \Gr(\mathcal{M}^{\theta})/ \sim_{SK}, \sim_h,$$ where $\Gr(\mathcal{M}^\theta)$ is the group completion of the monoid $\mathcal{M}^\theta$, and
     where $\sim_h$ is the equivalence relation that is generated by homotopies of tangential structures, i.e., 
     $$[\pi \colon E \to B, f\colon E\to X, \tau_{\pi} E \cong f^*\xi] \sim_h [\pi \colon E \to B, g\colon E\to X, \tau_{\pi}E \cong g^*\xi]$$ whenever there is a homotopy $f\simeq g$ through bundle maps over $B$.
 \end{defn}

 We now give a parametrized refinement of \cite[Theorem 4.4]{WIT} for $\theta$-structures, namely we show that the $SK$-group of $\theta$-manifolds coincides with the $K_0$-group of the scissors congruence $K$-theory of parametrized manifolds with $\theta$-structure. For $\theta$ as in \autoref{orientedex} and $X=\ast$, we recover the $SK$-groups for oriented manifolds with boundary from \cite{WIT}.

\begin{thm}\label{K0isSK}
Let $\theta=(B, p, \xi)$ be as in \autoref{thetastr}. There is a natural isomorphism
    $$K^\square_0(\Mtheta_{\Delta})\cong SK(\theta).$$
\end{thm}

\begin{proof}
Since $K^\square_0(\Mnfld^{\theta_n})$ is abelian for every $n \geq 0$, we have
$$K^\square_0(\Mnfld^{\theta_n})\cong \pi_1(|N_{\dotp} T_{\dotp} \Mnfld^{\theta_n}|) \cong H_1(|N_{\dotp} T_{\dotp} \Mnfld^{\theta_n}|).$$
Similarly, we can identify
$$K^\square_0(\Mtheta_{\Delta})\cong H_1(|N_{\dotp} T_{\dotp} \Mnfld^{\theta_{\sbt}}|).$$ 
Based on these identifications, the fact that $|N_{\dotp} T_{\dotp} \Mnfld^{\theta_n}|$ is connected for every $n$, and using the homology spectral sequence of the simplicial space $[n] \mapsto |N_{\dotp} T_{\dotp} \Mnfld^{\theta_n}|$, we conclude that 
$$\begin{tikzcd}
    K^\square_0(\Mtheta_{\Delta})\cong\mathrm{coeq} \big( K_0^\square(\Mnfld^{\theta_1})   \arrow[r, shift left=0.75ex, "d_0"]  \arrow[r, shift right=0.75ex, "d_1", swap]  & K_0^\square (\Mnfld^{\theta_0}) \big).
\end{tikzcd}$$
The quotient that results from the co-equalizer corresponds to a ``concordance" equivalence relation that identifies those $\theta_0$-manifolds that are at the endpoints of a parametrized $\theta_1$-manifold over $B \times \Delta_1$. As the underlying manifold bundles are isomorphic in this case, this relation simplifies to homotopy of $\theta_0$-structures, and so it reduces to the equivalence relation $\sim_h$ from \autoref{SKtheta}. 

By \cite[Theorem 3.1]{CKMZ}, the group $K^{\square}_0(\Mnfld^{\theta_0})$ is isomorphic to the free abelian group on isomorphism classes of $\theta_0$-manifolds, modulo the  $\square$-relations given by
  $[\varnothing] = 0$ and $[A] + [D] = [B] + [C]$ for every distinguished square 
  \begin{equation} \label{diagram: square}
    		\begin{tikzcd}[column sep=tiny, row sep=tiny]
    			A\arrow[rr, hook] \arrow[dd, hook] && B\arrow[dd, hook]\\
    			&\square &\\
    			C\arrow[rr, hook] && D
    		\end{tikzcd}
 \end{equation} 
  in $\Mnfld^{\theta_0}$. Therefore, 
\[
K^\square_0(\Mtheta_{\Delta}) \cong \Gr(\mathcal{M}^{\theta_0})/ \sim_{\square}, \sim_h.
\]

To show that $K^\square_0(\Mtheta_{\Delta})\cong SK(\theta)$, we are left to show that the square and homotopy relations imply the $SK$-relations on $\theta_0$-manifolds, and conversely that the $SK$ and homotopy relations  imply the square relations on $\theta_0$-manifolds.

Let $E, E'$ be two $\theta_0$-manifolds related by a cut-and-paste operation. Then by  \autoref{def: cut-and-paste of theta mnfds} there is a decomposition
$$E = E_1 \cup_{E_0} E_2 \ \text{ and } E' = E'_1 \cup_{E'_0} E'_2,$$
   \noindent where $E_0 \subseteq \partial E_i$, $E'_0 \subseteq \partial E'_i$ are subbundles of closed smooth manifolds for $i=1,2$, and  there are diffeomorphisms of $\theta_0$-manifolds 
 $$E_1 \cong E'_1 \ \text{ and } E_2 \cong E'_2$$
 which send $E_0$ to $E'_0$. We claim that the classes of $E$ and $E'$ in $K^{\square}_0(\Mtheta_\triangle)$ are equal. 

Let $C_1$, $C_2$ be (parametrized, closed) collars of $E_0$ in $E_1$ and $E_2$, respectively.  Let $i_1 \colon E_0 \times [-1,1] \to E_1\cup C_2$ and $i_2 \colon E_0 \times [-1,1] \to E_2\cup C_1$ be the maps that include the top half of the cylinder onto $C_1$ and the bottom half onto $C_2$ of the corresponding targets. These are $SK$-embeddings of $\theta_0$-manifolds in the sense of the \autoref{SKembbundle}. Then we have a distinguished square of $\theta_0$-manifolds

 \begin{center}
    		\begin{tikzcd}[column sep=tiny, row sep=tiny]
    			E_0 \times [-1,1] \arrow[rr, hook, "i_1"]\arrow[dd, hook, "i_2" swap] && E_1 \cup C_2 \arrow[dd, hook]\\
    			&\square &\\
    			E_2 \cup C_1 \arrow[rr, hook] && E_1 \cup_{E_0} E_2 = E,
    		\end{tikzcd}
\end{center}   
where the $\theta_0$-structure on $E_0 \times [-1, 1]$ is induced by either of the equal composite embeddings into $E$. 
Using the homotopy relation we can deform the $\theta_0$-structure on $E_1 \cup C_2$ up to homotopy, such that it agrees with the $\theta_0$-structure on $E_1$ along an isomorphism of parametrized manifolds $E_1 \cong E_1 \cup C_2$ that reparametrizes the two-sided collar. 
Thus $[E_1 \cup C_2]=[E_1]$, and similarly $[E_2\cup C_1]=[E_2]$. The relation given by distinguished squares implies
\[
[E_0 \times [-1,1]] + [E]=[E_1] +[E_2]\]
and analogously we get that \[[E'_0 \times [-1,1]] + [E'] = [E'_1] +[E'_2].
\]
As before, we may deform the $\theta_0$-structure on $E_0\times [-1,1]$ relative $E_0 \times \{0\}$ (resp. $E_0'\times [-1,1]$) so that the structure maps to $X$ are constant in the cylindrical direction. So we may assume that these cylinders are equipped with cylindrical $\theta_0$-structures. Then the diffeomorphism of $\theta_0$-manifolds $E_1\cong E_1'$ taking $E_0$ to $E_0'$ gives a diffeomorphism of the $\theta_0$-manifolds $E_0\times [-1,1]$ and $E_0'\times [-1,1]$ equipped with cylindrical $\theta_0$-structures. As a consequence, $[E_0\times [-1,1]] = [E_0'\times [-1,1]]$, and so we conclude $[E]=[E']$.
  
Conversely, suppose we are given a distinguished square in $\Mnfld^{\theta_0}$ as in \autoref{diagram: square}. Recall \autoref{def: intcomp} of interior and complement manifold bundles for $SK$-embeddings of manifold bundles. Define $N$ to be the manifold subbundle of $\partial A$ which maps into the interior of $B$. Let $A^c$ be the complement bundle of the $SK$-embedding $A \to B$. Let 
  $$E\coloneqq A^c \sqcup (N \times [0,1]),$$ where $N \times [0,1]$ is given a $\theta_0$-structure induced by a collar neighbourhood inclusion $N \times [0,1] \hookrightarrow A^c$ such that $N \times \{0\}$ is mapped to $N \subset A^c$. Let
  $$E'\coloneqq A \sqcup C.$$
Then we can construct two $\theta_0$-manifolds that both decompose into $E$ and $E'$:
$$(A^c \cup_N A) \sqcup ( (N \times [0,1])\cup_N C) \sim_{SK} (A^c \cup_N C) \sqcup ((N \times [0,1])\cup_N A).$$
Moreover, using similar arguments as above, 
$$(A^c \cup_N A) \sqcup ( (N \times [0,1])\cup_N C)\sim_h B \sqcup C$$ and 
$$(A^c \cup_N C) \sqcup ((N \times [0,1])\cup_N A) \sim_h  D \sqcup A.$$
Therefore, $[B\sqcup C] = [D \sqcup A]$ in $SK(\theta)$, so $[B]+[C]=[A]+[D]$ in $SK(\theta)$, as required.
\end{proof}

\section{The map from the cobordism category} \label{section: map from cob} In this section we define a map from the loop space of the classifying space of the cobordism category of manifolds with free boundary to the scissors congruence $K$-theory of manifolds. We construct this in the bivariant context of the parametrized cobordism category and scissors congruence $K$-theory of $\theta$-manifolds.

\subsection{The parametrized cobordism category} We recall the definition of the parametrized smooth cobordism category from \cite{RS17}. Instead of dealing with usual cobordisms between closed manifolds, we focus on a larger category $\Cobf (\theta)$ which is a version of the cobordism category that best approximates $K^\square(\Mtheta_\Delta)$. $\Cobf (\theta)$ is the full subcategory of the parametrized cobordism category $\Cob^{ \partial}(\theta)$ of manifolds with boundary (introduced in \cite{Genauer12} and used in \cite{RS14, RS17}) whose objects are closed $\theta$-manifolds. The morphisms of $\Cobf(\theta)$ are given by parametrized $\theta$-cobordisms that are allowed to have ``free'' boundary components not coming from the objects. Below we provide a detailed description.

\begin{figure}[h!]
\includegraphics[scale=.32]{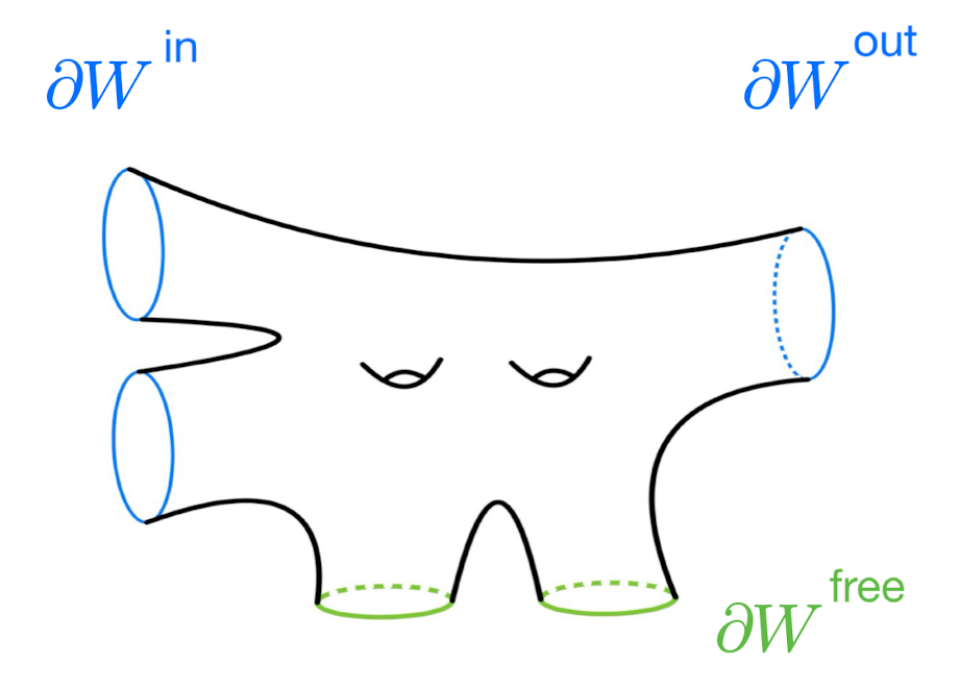}
\caption{Example of a cobordism with free boundary}
\centering
\end{figure}

Let $\theta=(B, p \colon X \to B, \xi \colon V \to X)$ be a $\theta$-structure as in \autoref{thetastr}, and let $d$ be the rank of the vector bundle $\xi$. The (non-topologized) category $\Cobf^{\delta}(\theta)$ of parametrized $\theta$-cobordisms with free boundary is defined as follows. 
\begin{defn} \label{def: category Cobf(theta)}
An \textit{object} in $\Cobf^{\delta}(\theta)$ is given by a quadruple $(E, \pi, a ,l)$, where
\begin{itemize}
    \item[(i)] $a \in \R$;
    \item[(ii)] $\pi \colon E \to B$ is a numerable fiber bundle of smooth closed $(d-1)$-dimensional manifolds, which is fiberwise smoothly embedded in 
    $B \times \{ a \} \times \R_{+} \times \R^{\infty}$, cylindrically near the collar;
    \item[(iii)] $l$ is a tangential $\theta$-structure, i.e., a bundle map $\epsilon \oplus \tau_{\pi}E \to \xi$ (fiberwise over $B$), where $\epsilon$ denotes the trivial $\R$-bundle, 
    $$\begin{tikzcd}
    	\epsilon \oplus \tau_{\pi}E \arrow[rr] \arrow[d] && V \arrow{d}{\xi} \\
    	E \arrow{rd}[swap]{\pi} \arrow[rr] && X \arrow{dl}{p} \\
            & B. &&
    	\end{tikzcd} $$

\end{itemize}
A \textit{morphism} in $\Cobf^{\delta}(\theta)$ consists of (i)-(ii) a numerable fiber bundle of compact smooth collared $d$-manifolds $$\pi \colon W \to B$$ 
fiberwise smoothly embedded in $B \times [a_0, a_1] \times \R_{+} \times \R^{\infty}$, $a_0 < a_1$, and cylindrically near the boundary. This is equipped with (iii) a tangential $\theta$-structure, given by a bundle map $$l_W \colon \tau_{\pi} W \to \xi$$ 
fiberwise over $B$ and cylindrical near the fiberwise boundary parts. The domain $M_0$ and the target $M_1$ of the morphism $W$ are given by the intersection of $W$ with  $B \times \{a_0\} \times \R_{+} \times \R^{\infty}$ and  $B \times \{ a_1 \} \times \R_{+} \times \R^{\infty}$, respectively. Composition of morphisms is defined by taking the union of embedded manifold bundles. This definition yields a non-unital category. We add formal units to turn it into a unital category.
\end{defn}

\begin{rem}
    \label{nonunital}
    The nerve of a non-unital category is a semisimplicial set, and its classifying space is defined as the geometric realization of this semisimplicial set. Alternatively, if we add formal units (see e.g. \cite{ebert2019semisimplicial}) to turn a non-unital category into a unital category, the nerve of the resulting category is a simplicial set which agrees with the image of the nerve of the non-unital category under the left adjoint to the forgetful map from simplicial sets to semisimplicial sets. This left adjoint formally adds in degeneracies to a semisimplicial set.
\end{rem}

 \begin{rem}
  If one additionally requires that morphisms $W$ don't have any boundary fiberwise apart from its source and target (so they are fiberwise embedded in $B \times [a_0, a_1] \times \R_{>0} \times \R^{\infty}$), then we would obtain the standard cobordism category, where no ``free'' boundary is allowed for the morphisms. It is a subcategory of $\Cobf^{\delta}(\theta)$ and we denote it by $\Cob^{\delta}(\theta)$.     
 \end{rem}

The parametrized cobordism category determines a bivariant theory $\theta \mapsto \Cobf(\theta)$ (see \cite{RS17}). Its simplicial thickening essentially encodes the natural topology of the cobordism category up to homotopy equivalence. The simplicial (``topological") $\theta$-cobordism category $\Cobf(\theta)$ is defined levelwise by
  \[
\Cobf(\theta)_n \coloneqq \Cobf^{\delta}(\theta_n),
  \]
where $\theta_n$ was introduced in \autoref{subsec: bivariant}. The (fat) geometric realization of the simplicial space 
$$[n] \mapsto B\Cobf^{\delta}(\theta_{\dotp})$$
is the classifying space of the parametrized $\theta$-cobordism category with free boundary
\[
B\Cobf(\theta) \coloneqq |B\Cobf^{\delta}(\theta)_{\dotp}|.
\]
Analogously, $B\Cob(\theta) \coloneqq |B\Cob^{\delta}(\theta)_{\dotp}|.$

\begin{exmp}
    In particular, for the $\theta$-structure from \autoref{orientedex} and $X=\ast$, the category $\Cob(\theta)$ recovers a simplicial version of the oriented $d$-dimensional cobordism category from \cite{GMTW}. We will denote this by $\Cob_d$.
\end{exmp}

\subsection{The idea of the bivariant transformation from the cobordism category}\label{idea} Before we give the rigorous definition of the bivariant transformation $f(\theta) \colon \Omega B\Cobf(\theta)\to K^\square(\Mtheta_\Delta)$ in \autoref{cobmapsec}, we start by discussing the idea of how this map is constructed, and explain some of the subtleties that arise when trying to make this precise. 
 
The construction of $f(\theta)$ results from the construction of a map between the unthickened models, natural in $\theta$, 
$$\sd N_{\sbt} \Cobf(\theta) \to N_{\sbt}^{\square}(\Mnfld^{\theta})$$
from the edgewise subdivision of the source to the squares construction for $\Mtheta$ (see \autoref{not:squares}). An $n$-simplex in the subdivision $\sd N_{\sbt} \Cobf(\theta)$ is a composite of $(2n+1)$ cobordisms, which we wish to send to an $n\times n$ square. Roughly, on 0-simplices, we map a cobordism $W$ to the corresponding manifold $W$. On 1-simplices, we map an 1-simplex given by a 3-composite of cobordisms to a $1\times 1$ square as depicted in \autoref{themap}, and similarly on higher simplices.

\begin{figure}[h!]
\centering
\includegraphics[scale=0.35]{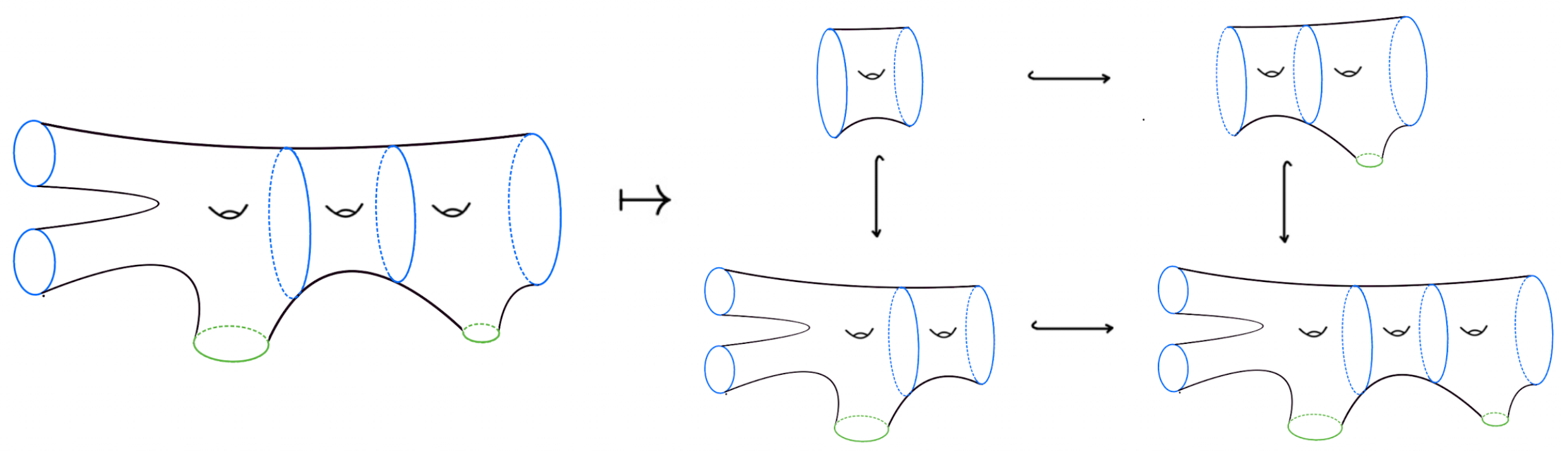} 
\caption{Idea of the map $\sd N_{\sbt} \Cobf(\theta) \to N_{\sbt}^{\square}(\Mnfld^{\theta})$ on the level of $1$-simplices}
\label{themap}
\end{figure}

However, a technical subtlety arises because we have added formal units to the cobordism category. Note that, while for a simplicial set $X_{\sbt}$, we have that $|\sd X_{\sbt}|\simeq |X_{\sbt}|$, this is not true for semisimplicial sets in general, as the following simple example illustrates.

\begin{exmp}
  Consider the 2-simplex as a semisimplicial set, i.e., a semisimplicial set with three 0-simplices, three one simplices, and one 2-simplex. Its edgewise subdivision consists only of three 0-simplices. The corresponding geometric realizations are a triangle and 3 points, which are not homotopy equivalent. Note that as a simplicial set, the 2-simplex has many degenerate simplices, and those contribute to the subdivision and ensure that the edgewise subdivision is equivalent to the original simplicial set.
\end{exmp}

So it is crucial that we add in the formal degeneracies and view $N_{\sbt} \Cobf(\theta)$ as a simplicial set since we need to subdivide it in order to create the map as depicted above. But then the technical question arises as to where to send the formal identities. This leads us to introducing some intermediary categories of $\varepsilon$-collared parametrized $\theta_k$-cobordisms in the next section which provide a canonical way to map the formal identities.

\subsection{The definition of the bivariant transformation from the cobordism category}\label{cobmapsec}

In this section we define the bivariant transformation for each $\theta=(B, p \colon X \to B, \xi \colon V \to X)$ (\autoref{thetastr})
$$f(\theta) \colon \Omega_{\varnothing} B\Cobf(\theta)\to K^\square(\Mtheta_\Delta)$$
from the looped parametrized cobordism category functor to the parametrized scissors congruence $K$-theory of manifolds  functor, using the idea from \autoref{idea}. We first introduce categories of $\varepsilon$-collared parametrized cobordisms in order to be able to define the map on the formal units.

Denote by $\Cobfe^{\delta}(\theta_k)$ the category of $\varepsilon$-collared parametrized $\theta_k$-cobordisms, $\epsilon > 0$. The $\varepsilon$-collar condition means that we only allow those morphisms $W$ which are fiber bundles of manifolds that are $\varepsilon$-cylindrical near the boundary. More precisely, in the notation of the \autoref{def: category Cobf(theta)}
\[
W \cap \big( B \times [a_0, a_0 + \epsilon] \times \R_{+} \times \R^{\infty} \big) = M_0 \times [a_0, a_0 + \epsilon],
\]
\[
W \cap \big( B \times [a_1-\varepsilon, a_1] \times \R_{+} \times \R^{\infty} \big) = M_1 \times [a_1-\epsilon, a_1].
\]

We view $\Cobfe^{\delta}(\theta_k)$ as a unital category by adding in formal identities as in \autoref{nonunital}. Intuitively, one could think of the identity morphisms in $\Cobfe^{\delta}(\theta_k)$ as cobordisms of length $0$. 

\noindent
\begin{con}\label{L}
We start by defining an \textit{enlargement map}, denoted by $L$, that takes a morphism $W$ in $\Cobfe^{\delta}(\theta_k)$ as input and outputs a $\theta_k$-manifold obtained by attaching cylinders fiberwise to the boundary components of $W$. Recall that a morphism in $\Cobfe^{\delta}(\theta_k)$ is given by a fiber bundle of compact smooth $d$-manifolds $\pi \colon W \to B$ embedded fiberwise in $B \times [a_0, a_1] \times \R_{+} \times \R^{\infty}$ and $\varepsilon$-cylindrically near the boundary. The map $L$ extends each fiber of $W$ by attaching cylinders of length $1$ to each incoming and outgoing boundary of the fiber, resulting in a new fiber bundle $L(W) \to B$ embedded fiberwise in $B \times [a_0-1, a_1+1] \times \R_{+} \times \R^{\infty}$ such that
  \[
  L(W) \cap \big( B \times [a_0, a_1] \times \R_{+} \times \R^{\infty} \big) =W,
  \]
  \[
  L(W) \cap \big( B \times [a_0-1, a_0] \times \R_{+} \times \R^{\infty} \big) =M_0 \times [a_0-1, a_0],
  \]
   \[
  L(W) \cap \big( B \times [a_1, a_1+1] \times \R_{+} \times \R^{\infty} \big) =M_1 \times [a_1, a_1+1].
  \] 
  We view $L(W)$ as a $\theta_k$-manifold. When $W$ is a formal identity of an object $M$, the $\theta_k$ manifold $L(W)$ is a cylinder on $M$ of length 2. 
  
  Let $W$ and $W'$ be two composable morphisms in $\Cobfe^{\delta}(\theta_k)$, then there is an $SK$-embedding of $\theta_k$-manifolds given by the inclusion 
  \[
  i \colon W \hookrightarrow W \cup W'.
  \]
  Let us define an induced  $SK$-embedding of $\theta_k$-manifolds
  \[
    L(i) \colon L(W) \to L(W \cup W').
  \]
  By construction $L(W)$ is obtained by the fiberwise attachments of unit cylinders to incoming and outgoing boundary components of the fibers of $W$. Since the incoming boundary of $W$ and of $W \cup W'$ coincide the map $i$ extends to the cylinders attached to the incoming boundary components via the identity. Since $W'$ is $\varepsilon$-cylindrical near its boundary, we define $L(i)$ to map the length $1$ cylinders attached to the outgoing boundary components of $W$ to the $\varepsilon$-collars of the incoming boundary components of $W'$. This construction yields the desired map $L(i)$.
 \end{con}   
\begin{figure}[h!]
\includegraphics[scale=0.4]{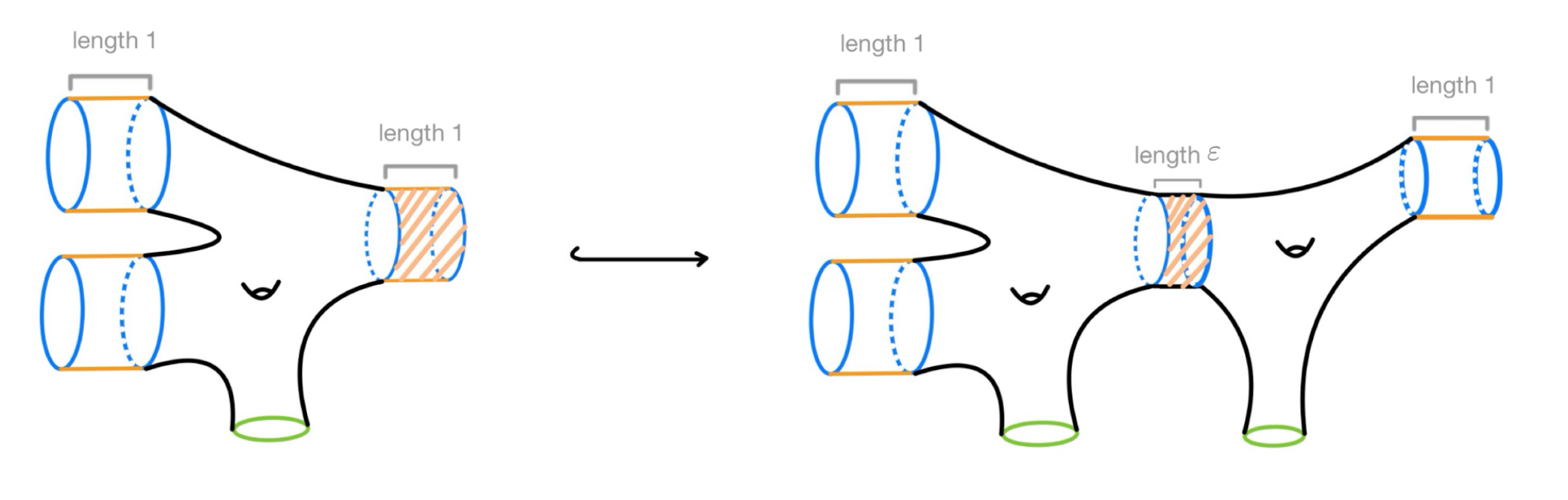}
\caption{Example of an $SK$-embedding $L(i)$ in a fiber}
\centering
\end{figure}

The proof of the following lemma is straightforward. 
\begin{lemma} \label{lem: functoriality of the enlargement map}
    Let $W_0, W_1, W_2$ be a sequence of three composable morphisms in $\Cobfe^{\delta}(\theta_k)$. Denote by
    \[
    i \colon W_0 \to W_0 \cup W_1,
    \]
    \[
    j \colon W_0 \cup W_1 \to W_0 \cup W_1 \cup W_2
    \]
 the obvious inclusions. Then  
    \[
    L(j) \circ L(i)=L(j \circ i).
    \]
\end{lemma}

For a given structure $\theta$, we now construct the components of a transformation at every simplicial level $k$ as a simplicial map from the edgewise subdivision of the nerve of $\Cobfe^{\delta}(\theta_k)$, 
\begin{equation}\label{mapfromBCob}\sd N_{\sbt}\Cobfe^{\delta}(\theta_k) \to N_{\sbt}^{\square}(\Mnfld^{\theta_k}),\end{equation} 
\noindent which is simplicial in $k$.

An $m$-simplex in the subdivision $\sd N_{\sbt}\Cobfe^{\delta}(\theta_k)$ is given by a string of $(2m+1)$ composable morphisms in $\Cobfe^{\delta}(\theta_k)$, which we index as follows
\[
W_m W_{m-1} \ldots W_1 W_0 W'_1 \ldots W'_{m-1} W'_m .
\]
Any of these morphisms could be formal identities, which we identify with the corresponding object (cobordism of length 0). For $0<i<m$ the $i$-th face map replaces the following parts of the string $W_{i+1}W_i$ and $W'_iW'_{i+1}$ by the corresponding composition of morphisms $W_{i+1} \cup W_i$ and $W'_i \cup W'_{i+1}$. The 0-th face map replaces the part of the string $W_1 W_0 W'_1$ by the composite of morphisms $W_1 \cup W_0 \cup W'_1.$ The $m$-th face map removes $W_m$ and $W'_m$ from the string.

 On $m$-simplices, we define the map from \autoref{mapfromBCob} by sending such a $(2m+1)$-string to the $m\times m$ square diagram of spaces with $W_0$ in the upper left corner, which successively includes the pieces to the right of it in the above composite horizontally, and the pieces to the left of it in the above composite vertically. Afterwards, we apply the enlargement map $L$ to each entry of the square to obtain an element in $N_{\sbt}^{\square}(\Mnfld^{\theta_k})$. For example, on 0-simplices, a morphism $W_0$ maps to $L(W_0)$, viewed as a manifold with boundary, which is a 0-simplex in $N_{\sbt}^{\square}(\Mnfld^{\theta_k})$. On 1-simplices, a composite $W_1W_0W'_1$ maps to the 1-simplex given by the square
$$ \begin{tikzcd}
     L(W_0) \arrow[r] \arrow[d] & L(W_0\cup W'_1) \arrow[d]\\
     L(W_1\cup W_0) \arrow[r] & L(W_1\cup W_0 \cup W'_1),
 \end{tikzcd}$$
 a 5-composite of morphisms similarly goes to a $2\times 2$ square, and so on. 
 We note that the $i$-th face map on $N_{\sbt}^{\square}(\Mnfld^{\theta_k})$ is given by deleting the $i$-th row and the $i$-th column, and degeneracies are given by identity maps on objects and repeating vertical maps on morphisms. Using \autoref{lem: functoriality of the enlargement map} it is easy to see that the constructed map is simplicial.  

Thus, for any $\theta$ and $k \geq 0$, the map we constructed induces a map on (fat) geometric realizations
$$|N_{\sbt} \Cobfe^{\delta}(\theta_k)|\simeq |\sd N_{\sbt} \Cobfe^{\delta}(\theta_k)|\to |N_{\sbt}^{\square}(\Mnfld^{\theta_k})|,$$
and upon realizing in the other simplicial direction (in $\theta_k$) and taking loops we obtain a map
\begin{equation}\label{map: f_epsilon}
f_\varepsilon(\theta) \colon \Omega_{\varnothing} B\Cobfe(\theta)\to K^\square(\Mtheta_\Delta).
\end{equation}

For $\epsilon_1>\epsilon_2$, there is an inclusion $B \Cobf_{\epsilon_1}^{\delta}(\theta_k)\to B \Cobf_{\epsilon_2}^{\delta}(\theta_k)$, and we have
    \begin{equation} \label{eq: BCob as a colimit of BCobe}
  B\Cobf^{\delta}(\theta_k) =  \colim_{\varepsilon \to 0} B \Cobfe^{\delta}(\theta_k)\simeq \hocolim_{\varepsilon \to 0} B \Cobfe^{\delta}(\theta_k).
    \end{equation}

The maps $f_\varepsilon(\theta)$ are natural in $\varepsilon > 0$ up to canonical homotopy, so there is an induced canonical map
\[
f(\theta) \colon \Omega_{\varnothing} B\Cobf(\theta)\to K^\square(\Mtheta_\Delta).
\]
Similarly, the (homotopy class of the) map $f(\theta)$ is natural in $\theta$ (in the homotopy category). It will suffice for our purposes here to work in the homotopy category (as target of these bivariant theories), however, we note that with some additional care involving choices of preferred homotopies, it is also possible to promote $f(\theta)$ to a homotopy coherent natural transformation or a natural transformation of bivariant theories taking values in the respective $\infty$-category.

Since $\Cob^{\delta}(\theta)$ is a subcategory of $\Cobf^{\delta}(\theta)$ the map constructed above restricts to a map from $\Omega_{\varnothing} B\Cob(\theta)$. To summarize, we have a composition of maps
\begin{equation} \label{map: precomposition of f with Bcob}
\Omega_{\varnothing} B\Cob(\theta) \to \Omega_{\varnothing} B\Cobf(\theta) \to K^\square(\Mtheta_\Delta).
\end{equation}

\section{The map to the algebraic $K$-theory of spaces}\label{sec:Atheory}
In this section we construct a map $g(\theta)$ from the parametrized squares $K$-theory of manifolds to bivariant $A$-theory. We then verify that its composition with the map $f(\theta)$ of \autoref{section: map from cob} recovers the parametrized Bökstedt-Madsen map $\tau(\theta)$ from \cite{RS17}.

\subsection{Bivariant $A$-theory} We briefly recollect the definition of the bivariant $A$-theory following \cite{WilliamsBiv, RS14, RS17}. This associates to a fibration $p \colon X \to B$ a $K$-theory spectrum $A(p)$ that is the $K$-theory of a Waldhausen category of retractive spaces over $X$ suitably related to $p$. Denote by $\mathcal{R}(X)$ the Waldhausen category of retractive spaces over $X$ and let $\mathcal{R}^{hf}(X) \subset \mathcal{R}(X)$ be the full subcategory of homotopy finite objects. The algebraic $K$-theory or  $A$-theory of a space $X$ is defined as the $K$-theory of the Waldhausen category $\mathcal{R}^{hf}(X)$ \cite{waldhausen2006algebraic}.

To define $A(p)$ we consider those retractive spaces over $X$ that define families of homotopy finite retractive spaces over the fibers of $p$, parametrized by the points of $B$. For technical reasons we assume that $B$ has the homotopy type of a CW-complex.

\begin{defn}
Let $p \colon X \to B$ be a fibration.  Let $\mathcal{R}^{hf}(p) \subset \mathcal{R}(X)$ be the full subcategory of those retractive spaces $X \xhookrightarrow[]{i} Y \xrightarrow[]{r} X$ such that:
\begin{enumerate}
    \item the composite $p \circ r \colon Y \to B$ is a fibration;
    \item for each $b \in B$ the fiber $(p \circ r)^{-1}(b)$ is homotopy finite as an object of $\mathcal{R}(p^{-1}(b))$.
\end{enumerate}
\end{defn}
The category $\mathcal{R}^{hf}(p)$ is a Waldhausen subcategory of $\mathcal{R}(X)$.

\begin{defn}
    The bivariant $A$-theory of $p \colon X \to B$ is defined to be the space
    \[
    A(p) \coloneqq K(\mathcal{R}^{hf}(p))= \Omega |wS_{\dotp} \mathcal{R}^{hf}(p)|.
    \]
\end{defn}
It was shown in \cite{RS14} that the assignment 
\[
\theta=(B, p, \xi) \mapsto A(p)
\]
gives a spectrum-valued bivariant theory.
Note that the bundle $\xi$ does not play a role in the definition of $A$.

\subsection{The parametrized Bökstedt-Madsen map} \label{subsection: BM-map} In \cite{RS17}, the Bökstedt-Madsen map from the loop space of the classifying space of the $d$-dimensional cobordism category  to $A(BO(d))$ was extended to allow arbitrary parametrized $\theta$-structures. Namely, for every $\theta=(B, p, \xi)$, the authors defined a bivariant transformation\footnote{Actually, in \cite{RS17} this map is defined on the even larger cobordism category of cobordisms with boundary $B\Cob_d^\partial$ studied by Genauer \cite{Genauer12}. We will use only the restriction of this to the category of cobordisms with free boundary $B\Cobf_d$.}
    \[
       \tau(\theta) \colon \Omega B \Cobf_d(\theta) \to A_{\Delta}(p),
    \]
    where $A_{\Delta}(p)$ denotes the simplicial thickening of $A(p)$ -- also referred to as the \textit{thick model} of bivariant $A$-theory. Note that the spectra $A_{\Delta}(p)$ and $A(p)$ are equivalent (see \cite[Section 3]{RS14}). 
    
This transformation $\tau(\theta)$ is induced by simplicial maps of simplicial categories 
   \[
      t_{\dotp}^m \colon N_{\dotp} \Cobf^{\delta}(\theta_m) \to w S_{\dotp} \mathcal{R}^{hf}(p_m), ~m \geq 0,
   \]
where the set $N_k \Cobf^{\delta}(\theta_m)$ is regarded as a category with only identity morphisms.   
An object in $N_k \Cobf^{\delta}(\theta_m)$ is given by the sequence of $k$ composable bundles of cobordisms over $B \times \Delta^m$:
  \[
    E[a_0, a_1],~ E[a_1, a_2], ~\ldots~ ,  ~E[a_{k-1}, a_k],
  \]
  which are endowed with a fiberwise tangential $\theta_m$-structure. It is mapped to the diagram of retractive spaces:
  \[
  Ar[k] \to \mathcal{R}^{hf}(p_m)
  \]
  \[
  (i \leq j) \mapsto E[a_i, a_j] \cup_{E(a_i)} (X \times \Delta^m),
  \]
  where $E[a_i, a_j]$ denotes the union of $E[a_i, a_{i+1}]$, \ldots, $E[a_{j-1}, a_{j}]$ and $E(a_i)$ is the incoming (resp. outgoing) boundary of $E[a_i, a_{i+1}]$ (resp. $E[a_{i-1}, a_i]$). This induces a bisimplicial map
  \[
    t_{*, \dotp}^m \colon N_{*} (N_{\dotp} \Cobf^{\delta}(\theta_m)) \to N_{*} (w S_{\dotp} \mathcal{R}^{hf}(p_m)), ~m \geq 0,
  \]
  which upon taking the geometric realization (as trisimplicial map) yields a bivariant transformation $\tau(\theta)$ (for more details, see \cite[Section 5]{RS17} and \cite[Section 5.1]{RS14}).
  
\subsection{The definition of the bivariant transformation to $A$-theory} \label{subsec:map_to_A-theory}
We construct a bivariant transformation $$g(\theta) \colon K^\square(\Mtheta_\Delta) \to A(p).$$ 
We will again use the thick model $A_{\Delta}(p)$. In addition, instead of the $S_{\dotp}$-construction, we will make use of the $T^{+}_{\dotp}$-construction, a refined version of the Thomason construction for Waldhausen $K$-theory. For a Waldhausen category $\C$, there is a simplicial (Waldhausen) category $(w)T^{+}_{\dotp} \C$ defined as follows: the category $w T^{+}_{k} \C$ has objects given by sequences of cofibrations 
  \[
    C=[C_0 \hookrightarrow C_1\hookrightarrow \dots \hookrightarrow C_k]
  \]
together with choices of quotients $C_{ij}=C_j/C_i$ for every $i \leq j$, such that $C_{ii}$ is required to be the zero object. The morphisms are natural transformations $C \to C'$ satisfying the condition that for every $i \leq j$ the induced map 
 \[
   C'_i \cup_{C_i} C_j \to C'_j
  \]
is a weak equivalence in $\C$. The forgetful map $w T^{+}_{k} \C \to w S_{k} \C$, which on objects forgets the objects $C_i$ and only retains the data of the quotients $C_{ij}$, as they form an $\mathrm{Ar}[k]$-diagram with the induced maps, is a homotopy equivalence (see \cite[p. 334]{waldhausen2006algebraic}). 

The map $g(\theta)$ is obtained from a bisimplicial map
\[
 g^m_{*, \dotp} \colon N_{*} T_\dotp \Mnfld^{\theta_m} \to      N_{*} (w T^{+}_{\dotp} \mathcal{R}^{hf}(p_m))
\]
which is natural for each $[m] \in \Delta^{\op}$, that is, it defines a trisimplicial map, after passing to the geometric realizations and loop spaces. 

For each $m \geq 0$, this bisimplicial map is given by the nerve of a functor, which is natural in $k \geq 0$,
\[
g^m_k \colon T_k \Mnfld^{\theta_m} \to w T^{+}_k \mathcal{R}^{hf}(p_m).
\]
The latter sends a sequence of $\SK$-embeddings of manifold bundles over $B \times \Delta^m$ endowed with fiberwise $\theta_m$-structures
  \[
    E_0 \hookrightarrow E_1 \hookrightarrow \dots \hookrightarrow E_k
  \]
to the sequence of cofibrations of spaces in  $\mathcal{R}^{hf}(p_m)$ 
  \[
    E_0 \sqcup \Xm \hookrightarrow E_1 \sqcup \Xm  \hookrightarrow \dots \hookrightarrow E_k \sqcup \Xm,
  \]
where $\Xm$ denotes $X \times \Delta^m$ and the choice of quotient $E_{ij}$ is given by the disjoint union of $\Xm$ with the quotient space $E_j/\im(E_i)$. This gives a simplicial map of simplicial categories
  \[
g^m_{\dotp} \colon T_{\dotp} \Mnfld^{\theta_m} \to w T^{+}_{\dotp} \mathcal{R}^{hf}(p_m),
  \]
inducing the desired bisimplicial map
\[
 g^m_{*, \dotp}  \colon N_{*} T_\dotp \Mnfld^{\theta_m} \to      N_{*} (w T^{+}_{\dotp} \mathcal{R}^{hf}(p_m)).
\]

  \begin{prop} \label{prop: decomposing BM-map as fg}
 For every $\theta = (B, p \colon X \to B, \xi \colon V \to X)$, the composition
       \[
 \Omega B \Cobf(\theta) \xrightarrow{f(\theta)} K^\square(\Mtheta_{\Delta}) \xrightarrow{g(\theta)} A(p)
      \]
agrees up to homotopy with the map $\tau(\theta) \colon \Omega B \Cobf(\theta) \to A(p)$.
\end{prop}

  \begin{proof}
We consider the following diagram

\begin{equation}\label{simplicialdiag}
\begin{tikzcd} 
    \sd N_{*}\Cobfe^{\delta}(\theta_m) \arrow{r}{(f_{\epsilon}^m)_{*}} \arrow{d}[swap]{\text{last-vertex}} & \mathrm{diag}(N_{*} T_\dotp \Mnfld^{\theta_m}) \simeq N_{*} T_\dotp \Mnfld^{\theta_m} \arrow{r}{g^m_{*, \dotp}} & N_{*} (w T_{\dotp}^{+}\mathcal{R}^{hf}(p_m)) \arrow{dd}{\text{quotient}}\\
     N_{\dotp}\Cobfe^{\delta}(\theta_m) \arrow{d}[swap]{\mathrm{diag}}&&\\
	N_{*} N_{\dotp}\Cobfe^{\delta}(\theta_m)\arrow{rr}{(t_\varepsilon^m)_{*, \dotp}} && N_{*}(w S_{\dotp} \mathcal{R}^{hf}(p_m)).
	\end{tikzcd}
 \end{equation} 
where: 
  \begin{itemize}
\item  $f_{\epsilon}^m$ is the map that yields the map $f_\epsilon(\theta_m)$ (\autoref{cobmapsec}); 
      \item $t_\varepsilon^m$ indicates the restriction of the map $t^m$  (\autoref{subsection: BM-map}) to the subcategory $\Cobfe^{\delta}(\theta_m)$ of $\epsilon$-cylindrical cobordisms;
      \item $\mathrm{diag}$ denotes the map that sends a $k$-simplex in the source to the $(k, k)$-bisimplex obtained by the extension via the $k$ copies of the identity morphism in the second simplicial direction.  
  \end{itemize}  

We claim that the diagram commutes up to a preferred homotopy. We start with a $k$-simplex in $\sd N_{*}\Cobfe^{\delta}(\theta_m)$, which is given by a sequence of $(2k+1)$ bundles of composable cobordisms over $B \times \Delta^m$ with $\theta_m$-structure:
  \[
    E[a_0, a_1],~ E[a_1, a_2], ~\ldots~ ,  ~E[a_{2k}, a_{2k+1}].
  \]
Under the last-vertex map it is mapped to the following sequence of $k$ composable bundles of cobordisms in $N_k\Cobfe^{\delta}(\theta_m)$:
  \[
    E[a_{k+1}, a_{k+2}],~ E[a_{k+2}, a_{k+3}], ~\ldots~ ,  ~E[a_{2k}, a_{2k+1}].
  \]
The diagonal map sends it to a $(k, k)$-bisimplex given by extending the previous object of the category $N_k\Cobfe^{\delta}(\theta_m)$ by $k$ copies of the identity morphism in the $*$-simplicial direction.
Under the map $t^m_{\epsilon}$, it is mapped to a $(k, k)$-bisimplex, whose $*$-simplicial direction is given by $k$ copies of the identity morphism at the object in $w S_{k} \mathcal{R}^{hf}(p_m)$ defined by
  \[
  \mathrm{Ar}[k] \to \mathcal{R}^{hf}(p_m)
  \]
  \[
  (i \leq j) \mapsto E[a_{k+i+1}, a_{k+j+1}] \cup_{E(a_{k+i+1})} \Xm.
  \]
Here the notation $E[a,b]$ indicates the composite of the cobordisms between reference points $a$ and $b$, and we write $\Xm = X \times \Delta^m$ for brevity. 
  
 Let us now track the image of the $k$-simplex from  $\sd N_{*}\Cobfe^{\delta}(\theta_m)$ under the second path in the diagram. The image under the map $(f_\varepsilon^m)_*$ of a $k$-simplex that is determined by composable morphisms/bundles of $\theta_m$-cobordisms:
   \[
    E[a_0, a_1],~ E[a_1, a_2], ~\ldots~ ,  ~E[a_{2k}, a_{2k+1}]
  \]  
 is given by the $(k,k)$-bisimplex in $N_{k} T_k \Mnfld^{\theta_m}$: 
 \[
\begin{tikzcd}[column sep=tiny, row sep=tiny]
L (E[a_k, a_{k+1}]) \arrow[rr, hook]\arrow[dd, hook] && L(E[a_k, a_{k+2}]) \arrow[rr, hook]\arrow[dd, hook] &&\cdots\arrow[rr, hook] && L(E[a_{k}, a_{2k+1}]) \arrow[dd, hook]\\
& \square && \square && \square\\
L(E[a_{k-1}, a_{k+1}]) \arrow[rr, hook]\arrow[dd, hook] && L(E[a_{k-1}, a_{k+2}]) \arrow[rr, hook]\arrow[dd, hook] &&\cdots\arrow[rr, hook] && L(E[a_{k-1}, a_{2k+1}]) \arrow[dd, hook]\\
& \square &  & \square  & & \square \\
\vdots\arrow[dd, hook] && \vdots\arrow[dd, hook] && \ddots &&\vdots\arrow[dd, hook]\\
& \square & & \square  & & \square \\
L(E[a_{0}, a_{k+1}]) \arrow[rr, hook] && L(E[a_0, a_{k+2}]) \arrow[rr, hook] &&\cdots\arrow[rr, hook] && L(E[a_0, a_{2k+1}]).
\end{tikzcd}
\]
In other words, it has $L (E[a_k, a_{k+1}])$ in the upper left corner, where $L$ is the enlargement map that adds cylinders of length one to the boundaries (\autoref{L}), and successively includes horizontally the pieces to the right of the original sequence of morphisms, and vertically the pieces to the left of the sequence.
Under the map $g^m_{*, \dotp}$ this square is mapped to a $(k,k)$-bisimplex in $N_{*} (w T_{\dotp}^{+}\mathcal{R}^{hf}(p_m))$ given by the square with $(i,j)$-entry being $L(E[a_{k-i}, a_{k+j}]) \sqcup \Xm$, and with the choices of quotients given by the corresponding quotient spaces. Lastly, under the quotient map, this square diagram is mapped to the quotient by the first column diagram. To describe this more explicitly, let $\Lr$ denote the operation of an \textit{enlargement on the right}, defined similarly as in \autoref{L}, but with cylinders of length $1$ being attached only to the outgoing boundary of a cobordism. Then the first row of the resulting diagram in $wS_k \mathcal{R}^{hf}(p_m)$ is given by

 \[
\begin{tikzcd}[column sep=tiny, row sep=tiny]
\Lr (E[a_{k+1}+\epsilon, a_{k+2}]) \cup_{E(a_{k+1}+\epsilon)} \Xm \arrow[rr, hook] && \Lr (E[a_{k+1} +\epsilon, a_{k+3}]) \cup_{E(a_{k+1}+\epsilon)} \Xm \ar[rr, hook] && ... \\
{\color{white}...................................................} ... \ar[rr, hook] &&   \Lr (E[a_{k+1} + \epsilon, a_{2k+1}]) \cup_{E(a_{k+1} +\epsilon)} \Xm 
\end{tikzcd}
\]
since the following diagram is a pushout
\[
\begin{tikzcd}
L(E[a_k, a_{k+1}]) \sqcup \Xm \arrow[r, hook] \arrow[d] & L(E[a_k, a_{k+i}]) \sqcup \Xm \arrow[d]\\
\Xm \arrow[r, hook] & \Lr (E[a_{k+1} +\epsilon, a_{k+i}]) \cup_{E(a_{k+1}+\epsilon)} \Xm.
\end{tikzcd}
\]
Similarly, the resulting object in the $\bullet$-direction is the diagram of retractive spaces 
  \[
  \mathrm{Ar}[k] \to \mathcal{R}^{hf}(p_m)
  \]
  \[
  (i \leq j) \mapsto \Lr (E[a_{k+i+1} + \epsilon , a_{k+j+1}]) \cup_{E(a_{k+i+1} + \epsilon)} \Xm.
  \]
and we note that the induced vertical maps (in the $\ast$-direction) become identities after taking the quotient by the first column. 
  
Thus, we have identified the images of a given $k$-simplex under the two composite maps, and we observe that these are related by a natural weak equivalence. This yields a natural simplicial homotopy and completes the proof that the diagram from \autoref{simplicialdiag} commutes up to a preferred homotopy. Upon passing to the colimits as $\epsilon\to 0$, we obtain the desired homotopy commutative diagram, depicted informally as follows:
    \[ \begin{tikzcd}
	\Omega B \Cobf(\theta) \arrow{r}{f(\theta)} \arrow{d}[swap]{\text{(subdivision)} \ \simeq } & K^\square(\Mtheta_\Delta) \arrow{r}{g(\theta)} & A_{\Delta}(p) \arrow{d}{(T^+ \to S) \ \simeq}\\
	\Omega B\Cobf(\theta) \arrow{rr}{\tau(\theta)} && A_{\Delta}(p).
	\end{tikzcd}
	\]
  \end{proof}

\begin{rem}
Given $\theta = (\ast, X, \xi \colon V \to X)$, we do not know whether the map $g(\theta) \colon K^\square(\Mtheta_{\Delta}) \xrightarrow{g(\theta)} A(X)$ factors through the unit map $\eta_X \colon Q(X_+) \to A(X)$. 
The corresponding factorization for the map $\tau(\theta)$ was shown in \cite{RS14, RS17} (see also \cite{RS20}) and was used to obtain a new proof of a refined bivariant version of the Dwyer-Weiss-Williams index theorem.  
\end{rem}

\section{Comparing the $K$-theory of manifolds and the cobordism category}\label{sec: pi_0 computation}

In this section, we specialize to the comparison map 
$$\pi_0 \Omega B\Cobf(\theta) \to K^\square_0(\Mtheta_\Delta)$$
induced by $f(\theta)$ (\autoref{cobmapsec}) for the specific $\theta$-structure from  \autoref{orientedex}, namely the case of $d$-dimensional oriented manifolds. 

\begin{notation}
    In this case, we will use the usual notation $\Cob_d$ and $\Cobf_d$ for the simplicial categories of $d$-dimensional oriented cobordisms, and oriented cobordisms with free boundary, respectively, and we will use the notation $\Mnfld^{\partial, d}$ and $K^\square(\Mnfld^{\partial,d}_\Delta)$ for the squares category and the scissors congruence $K$-theory of $d$-dimensional oriented smooth compact manifolds, respectively (i.e., the simplicial thickening of $K^\square(\Mnfldbd_d)$ from \cite{WIT}). We will also write $f_d \colon \Omega B\Cobf_d\to K^\square(\Mnfld^{\partial,d}_\Delta)$ for the map constructed in  \autoref{cobmapsec}.
\end{notation}

The group $\pi_0\Omega B\Cob_d=\pi_1 B\Cob_d$ is isomorphic to the controllable scissors congruence group of closed oriented $d$-manifolds, $SKK_d$, recalled in \autoref{def: SKK} below. On the other hand, as shown in \cite[Theorem 6.2]{HRS24} (or as a particular case of \autoref{K0isSK}), $\pi_0K^\square(\Mnfld^{\partial,d}_\Delta)\cong SK_d^\partial$ is the scissors congruence group of $d$-manifolds with boundary. So we cannot expect the map $\Omega B\Cob_d\to K^\square(\Mnfld^{\partial,d}_\Delta)$, defined by \autoref{map: precomposition of f with Bcob}, to be an equivalence. However, we ask the question of whether the map we constructed on the larger category 
of cobordisms with free boundaries (and its variations for different $\theta$-structures) is a homotopy equivalence. 

\begin{quest}
Is the map $f_d \colon \Omega B\Cobf_d\to K^\square(\Mnfld^{\partial,d}_\Delta)$ (and its parametrized versions for different $\theta$-structures) a homotopy equivalence?
\end{quest}

As a first step towards a positive answer to this question, in this section we carry out the computation of $\pi_0$ for the oriented case, showing that we indeed get an isomorphism in this case. The unoriented case also follows similarly. 

\begin{thm}\label{mainpi0thm}
The map $f_d \colon \Omega B\Cobf_d\to K^\square(\Mnfld^{\partial,d}_\Delta)$ induces a $\pi_0$-isomorphism.
\end{thm}

The proof of \autoref{mainpi0thm} will occupy the remainder of the section and will be divided into different subsections.

\subsection{The map from $SKK$ to $SK$} The computation of $\pi_0 \Omega B\Cobf_d=\pi_1 B\Cobf_d$ will build on classical results about $SK$ and $SKK$ groups, and will crucially use the quotient map $SKK_d\to SK_d$. After reviewing the definition of the latter and collecting the relevant short exact sequences from \cite{KKNO}, we compute the kernel of the map $SKK_d\to SK_d$. This result is essential for our computation of $\pi_1$ of the cobordism category with free boundaries.

Classically, there is a more refined relation than that of cutting and pasting, called $SKK$ (``scheiden und kleben, kontrollierbar"=``controllable cutting and pasting") in which we keep track of the gluing diffeomorphisms.

\begin{defn} \label{def: SKK} The group $SKK_d$ is the quotient of the group completion of the monoid of diffeomorphism classes of closed, oriented $d$-manifolds under disjoint union, by the $\SKK$-equivalence relation that is generated by
\[ [M_1 \cup_{\phi} \bar{M'_1}] \;  - \;  [M_1 \cup_{\psi} \bar{M'_1}] \; \;  \sim \;  \;  [M_2 \cup_{\phi}\bar{M'_2}] \; -\; [M_2 \cup_{\psi}\bar{M'_2}],\]
 \noindent where $M_1, M_1'$ and $M_2, M_2'$ are compact oriented manifolds such that $\partial M_1=\partial M_2$ and $\partial M_1'=\partial M_2'$, and $\phi, \psi \colon \partial M_1 \to \partial M_1'$ orientation preserving diffeomorphisms.
\end{defn}

The $\SKK$-groups have been interpreted as Reinhardt vector field bordism groups \cite{KKNO} and have also been shown to be isomorphic to 
$\pi_0$ of the Madsen-Tillman spectrum $MTSO(d)$, or equivalently to $\pi_1(B\Cob_d)$  (see, e.g., \cite{Ebert}, \cite{bokstedt2014geometric}).  

The $SKK$-invariants for closed orientable manifolds are completely classified in \cite[\S 4]{KKNO}: they are given by the Euler characteristic, the Kervaire semicharacteristic and the cobordism class. Therefore, in the group $SKK_d$, two classes $[M]_{SKK}=[N]_{SKK}$ are equal exactly when $M$ and $N$ have the same Euler characteristic, semicharacteristic and cobordism class. Note that the $SKK$-equivalence relation is finer than the $SK$-relation, and hence there is a natural quotient map 
\[q \colon SKK_d\to SK_d,\] 
which sends $[M]_{SKK}$ to $[M]_{SK}$. Note that $SK_d$ here denotes the cut-and-paste group of \emph{closed} oriented $d$-manifolds.

We further recall the fundamental short exact sequences established in \cite{KKNO}. Let $\SKbar_d$ be the quotient of $\SK_d$ by the cobordism relation, and let $\Omega_d$ be the cobordism group of $d$-dimensional closed oriented manifolds. 

\begin{thm}\cite[Theorems 1.1, 1.2, 4.2]{KKNO} \label{thm: SES}
\leavevmode
 \begin{enumerate}[label=(\roman*), ref=\thethm(\roman*)]
     \item \label{SES1} Let $I_d$ denote the subgroup of $\SK_d$ generated by the sphere $S^d$. Then there is a split short exact sequence 
     \[
     0 \to I_d  \to \SK_d \to \SKbar_d \to 0 
     \ \ \ \text{and} \ \ \ \ \ 
 I_d \cong \begin{cases}
   \Z, & ~d \text{~even}\\
   0, & ~d \text{~odd}.\\
\end{cases} 
      \]
 \item \label{SES3} Let $J_d$ denote the subgroup of $\SKK_d$ generated by the sphere $S^d$. Then there is a split short exact sequence
 \[
 0 \to J_d \to \SKK_d \to \Omega_d \to 0
 \ \ \ \text{and} \ \ \ \ \ 
J_d \cong \begin{cases}
   \Z, & ~d \text{~even}\\
   \Z/2\Z,& \text{$d$ = 1 mod 4}\\
   0, & \text{$d$ = 3 mod 4.}   
\end{cases}
 \]
For $d=4k+1$ the splitting is given by the Kervaire semicharacteristic $\kappa$.
 \item \label{SES2} Let $F_d$ denote the subgroup of $\Omega_d$ generated by all mapping tori. Then there is a short exact sequence 
 \[
 0 \to F_d  \to \Omega_d \to \SKbar_d \to 0.
 \]
 \end{enumerate}
\end{thm}

We use these short exact sequences to describe the kernel of the quotient homomorphism $q\colon \SKK_d \to \SK_d$.

\begin{prop}\label{SKtoSKK}
Let $T_d$ denote the subgroup of $\SKK_d$ generated by all mapping tori. Then there is a short exact sequence
    \[
0 \to T_d \to \SKK_d \xrightarrow{q} \SK_d  
\to 0.
\]
\end{prop}
\begin{proof}
Note that $T_d$ is  contained in the kernel of $q$ since for any mapping torus both the Euler characteristic and the signature vanish.  
From \autoref{thm: SES}(ii), we obtain a short exact sequence
 \[
 0 \to J_d / (J_d \cap T_d) \to \SKK_d/ T_d \to \Omega_d /F_d \to 0.
 \]
Since $q$ is trivial on $T_d$, it descends to an induced homomorphism
 \[
 \bar{q} \colon \SKK_d/T_d \to \SK_d, 
 \]
 whose restriction to the subgroup $J_d / (J_d \cap T_d)$ yields a homomorphism
 \[
 \tilde{q} \colon J_d / (J_d \cap T_d) \to I_d.
 \]
In this way, we obtain a map of short exact sequences (see \autoref{thm: SES}(i))
\[ 
\begin{tikzcd}
0 \ar[r] & J_d / (J_d \cap T_d) \ar[r] \ar{d}{\tilde{q}} & \SKK_d/ T_d \ar[r] \ar{d}{\bar{q}} & \Omega_d /F_d \ar[r] \ar{d}{\cong} & 0\\
 0 \ar[r] & I_d  \ar[r] & \SK_d \ar[r] & \SKbar_d \ar[r] & 0,
	\end{tikzcd}
    \]
\noindent where the isomorphism on the right follows from \autoref{thm: SES}(iii). 

We claim that the map $\tilde{q} \colon J_d / (J_d \cap T_d) \to I_d$ is an isomorphism. When $d \neq 4k+1$, the computation of the groups $I_d$ and $J_d$ from \autoref{thm: SES} implies that the quotient map $J_d \to I_d$ is an isomorphism, and hence $\tilde{q}$ is so, too. For $d=4k+1$ observe that $\CP^{2k} \times S^1$ is a mapping torus in a trivial way and hence reperesents an element in $T_d$. By the K\"unneth theorem $b_i(\CP^{2k} \times S^1)=b_i(\CP^{2k})+b_{i-1}(\CP^{2k})$ and therefore the semicharacteristic is $\kappa(\CP^{2k} \times S^1)=1 \in \Z/2$. Since $S^d$ also has semicharacteristic $1$ and its cobordism class is $0$ we conclude that $S^d$ is $SKK$-equivalent to $\CP^{2k} \times S^1$, which implies $J_d/(J_d \cap T_d)=0.$ Therefore, for $d=4k+1$ the map $\tilde{q}$ is an isomorphism, which concludes the proof.
 \end{proof} 

\subsection{Computation of $\pi_0\Omega B\Cobf_d$} We start by computing $\pi_0 \Omega B\Cobf_d$, and we then show that this is indeed isomorphic to the scissors congruence group via the map to the $K$-theory of manifolds that we constructed in \autoref{cobmapsec}.

\begin{prop}\label{isom}
   There is an isomorphism $\pi_0 \Omega B\Cobf_d \cong \SK^{\partial}_d$.
\end{prop}
\begin{proof}
Let $\Cob_{d}$ denote the subcategory of $\Cobf_d$ with the same objects and morphisms without free boundary, and $\End_{d-1}^{\varnothing}$ its full subcategory on one object, the empty manifold.  Note that the following diagram is a pullback square of (weakly unital) locally fibrant simplicial categories and (weakly unital) functors
\[ 
\begin{tikzcd}
\Cobf_{d}  \arrow[r,  "\partial"]\arrow[d] & \End_{d-1}^{\varnothing}
\arrow{d}{} \\
\Cob_{d}^{\partial} \arrow{r}{\partial}  & \Cob_{d-1}
	\end{tikzcd}
    \]
where $\partial$ is defined by taking the (horizontal) boundary. By \cite[esp. Theorem 2.11 and Lemma 4.4]{Steimle21} (cf. \cite{Genauer12}), we obtain a homotopy pullback square after passing to the classifying spaces of these categories. (Note that \cite{Steimle21} 
works with the non-unital versions of these categories, but this does
not alter the homotopy types of their classifying spaces.)
Moreover, the upper row gives rise to a homotopy fiber sequence
\begin{equation} \label{seq: Cobf Genauer type sequence}
 B\Cob_d \to B\Cobf_d \xrightarrow{B\partial} B \End_{d-1}^{\varnothing},   
\end{equation}
where the first map is induced by the inclusion functor. Consider the following part of the long exact sequence of homotopy groups
\[
\pi_2 B \End_{d-1}^{\varnothing} \xrightarrow{\alpha} \underbrace{\pi_1 B\Cob_d}_{\cong \SKK_d} \to \pi_1 B\Cobf_d  \xrightarrow{\beta} \underbrace{ \pi_1 B \End_{d-1}^{\varnothing}}_{\cong \Gr(\mathcal{M}_{d-1})} \xrightarrow{\gamma} \underbrace{\pi_0 B\Cob_d}_{\cong \Omega_{d-1}},
\]
where $\Gr(\mathcal{M}_{d-1})$ denotes the Grothendieck group of the monoid $\mathcal{M}_{d-1}$ of diffeomorphism classes of oriented closed smooth $(d-1)$-dimensional manifolds, with respect to disjoint union. 

Under the identifications $\pi_1 B \End_{d-1}^{\varnothing} \cong \Gr(\mathcal{M}_{d-1})$ and $\pi_0 B\Cob_d \cong \Omega_{d-1}$, the map $\gamma$ is given by sending the diffeomorphism class of a manifold to its bordism class. Thus, the kernel of $\gamma$ is precisely the subgroup $C_{d-1}$ of nullbordant manifolds. Note that $C_{d-1}$ is a free abelian group and agrees with the image of 
$\beta$. So we obtain a split short exact sequence
\[0\to \SKK_d/ \im ~\alpha \to \pi_1 B\Cobf_d \xrightarrow{\beta} C_{d-1}\to 0. \]

By \autoref{alphaim} below (whose proof we defer to \autoref{lemmaproof}), the image of $\alpha$ is given by the subgroup $T_d$ of $SKK_d$ generated by mapping tori, and by \autoref{SKtoSKK}, we have a canonical isomorphism $\SKK_d/T_d \cong SK_d$. As a consequence, we conclude 
\[
\pi_1 B\Cobf_d \cong (\SKK_d/ \im ~\alpha) \oplus C_{d-1} \cong \SK_d \oplus C_{d-1}  \cong \SK^{\partial}_d,
\]
where the last isomorphism is shown in \cite[Theorem 2.10]{WIT}.
\end{proof}

\begin{proof}[Proof of \autoref{mainpi0thm}]
  Both $\pi_0(\Omega B\Cobf_d)$ and $K_0^\square(\Mnfld^{\partial,d}_\Delta)$ are isomophic to $\SK^{\partial}_d$. We need to show that the isomorphism from \autoref{isom} agrees with the map $\pi_0(f_d) \colon \pi_0(\Omega B\Cobf_d) \to K^{\square}_0(\Mnfld^{\partial, d}_{\Delta})$. For this, we will need to refer back to the explicit identification of $K_0$ of a squares category in terms of generators and relations from \cite{CKMZ}.
  
  An oriented compact smooth $d$-dimensional manifold $W$, viewed as a cobordism from $\varnothing$ to $\varnothing$, determines a loop in $B\Cobf_d$. If the manifold $W$ has boundary components, they are all treated as free boundary components. The loop in $B\Cobf_d$ that is determined by $W$ is identified with the class of $[W] \in \SK^{\partial}_d$ under the isomorphism $\pi_1 B\Cobf_d \cong \SK^{\partial}_d$ of \autoref{isom}. 
  
  Let us now track the image of this loop $W$ in the space $K^\square(\Mnfld^{\partial,d}_\Delta)$ under the map $f_d$. The map $f_d$ is defined on the subdivision of $N_{\dotp} \Cobf_d$. At the level of geometric realizations, a loop in $B \Cobf_d$, represented by a cobordism $W$ with incoming boundary $\varnothing$ and outgoing boundary $\varnothing$, corresponds to the following two 1-simplices in $\sd N_{\dotp} \Cobf_d$: the first 1-simplex is given by the string of three composable cobordisms
 \[
 \varnothing \varnothing W ,
\]
where $\varnothing$ denotes a formal identity, and the second 1-simplex is given by the string
 \[
  W \varnothing \varnothing.
 \]
 By the construction of the map $f_d$, these two 1-simplices are mapped to the two 1-simplices in $N^{\square}_{\dotp} \Mnfld^{\partial,d}$ represented by the following two $1 \times 1$ squares:
\[
\begin{tikzcd}
     \varnothing \arrow[r, hook] \arrow[d, equal] & L(W) \arrow[d, equal]\\
     \varnothing \arrow[r, hook] & L(W),
\end{tikzcd}
\quad
 \begin{tikzcd}
     \varnothing \arrow[r, equal] \arrow[d, hook] & \varnothing \arrow[d, hook]\\
     L(W) \arrow[r, equal] & L(W).
\end{tikzcd}
\]  
  
By the proof of \cite[Theorem 3.1]{CKMZ} that identifies $K_0$ of a squares category in terms of generators and relations, and by \cite[Theorem 4.4]{WIT} where this group was identified with $\SK^{\partial}_d$, it follows that these two squares determine a loop 
in $|N^{\square}_\dotp \Mnfld^{\partial, d}|$ representing $[L(W)] \in \SK^{\partial}_d \cong K_0^\square(\Mnfld^{\partial,d}_\Delta)$. 
The manifold $L(W)$ is isomorphic to $W$. Therefore, the image of the loop represented by the cobordism $W$ is mapped under $\pi_0(f_d)$ to 
\[
[L(W)]=[W] \in \SK^{\partial}_d \cong K_0^\square(\Mnfld^{\partial,d}_\Delta).
\]
Hence this identification together with \autoref{isom} yield  \autoref{mainpi0thm}.
\end{proof}

 It remains to prove the claim of \autoref{alphaim}, which was used in the proof of \autoref{isom}.

\subsection{The image of $\alpha$ in $SKK_d$}\label{lemmaproof}
The goal of this section is to identify the image of the homomorphism $\alpha \colon \pi_2 B \End_{d}^{\varnothing} \xrightarrow{\alpha} \SKK_{d+1}$ that was used in the proof of  \autoref{isom} (note the change of the index).

Let $M$ be a closed oriented smooth $d$-dimensional manifold. We recall that there is a canonical map (see, e.g., \cite{RS14})
\[
\lambda_M \colon B \Diff(M) \to \Omega_{\varnothing} B \Cob_{d}.
\]
Here $B \Diff(M)$ is defined as (a simplicial version of) the space $\Emb (M, \R^{\infty})/ \Diff(M)$, so a point in $B\Diff(M)$ can be viewed as a cobordism/morphism in $\Cob_d$ from $\varnothing$ to $\varnothing$. Then the map $\lambda_M$ is essentially given by sending each point in $B\Diff(M)$ to the loop in $B \Cob_{d}$ (based at $\varnothing$) that is determined by the associated endomorphism of $\varnothing$. Note that the map $\lambda_M$ factors through $\Omega_{\varnothing} B \End^{\varnothing}_{d}.$

We also review the (parametrized) Pontryagin-Thom construction that gives a description of the composite map
$$B\Diff(M) \xrightarrow{\lambda_M} \Omega B\Cob_d \xrightarrow{\sim} \Omega^{\infty} MTSO(d)$$ (see \cite{GMTW, ebert2013vanishing, RS14}). 
	
Let $B_n(M) \coloneqq \Emb(M, \R^{d-1+n})/ \Diff(M)$ and $E_n(M) \coloneqq \Emb(M, \R^{d-1+n}) \times_{\Diff(M)} M$. The projection $p_n \colon E_n(M) \to B_n(M)$ defines a smooth oriented $M$-bundle and there is a natural fiberwise embedding:
\[ \begin{tikzcd}
E_n(M)  \arrow[hookrightarrow]{r}{j_n}
\arrow{d}[swap]{p_n} & B_n(M) \times \R^{d+n} \arrow{ld}{proj}\\
B_n(M). & 
\end{tikzcd}
\]
Denote by $\nu(p_n)=-T_{p_n} E_n(M)$ the fiberwise normal bundle of the embedding $j_n$. Recall that $T_{p_n} E_n(M)$ is the vertical tangent bundle of $p_n$,
\[
T_{p_n} E_n(M) = \Emb(M, \R^{d-1+n}) \times_{\Diff(M)} TM \to \Emb(M, \R^{d-1+n}) \times_{\Diff(M)} M=E_n(M).
\] 
Let $U$ be a tubular neighbourhood of the image of $j_n$.  The Pontryagin-Thom collapse map
\[
(B_n(M))_{+} \wedge S^{n+d} \to \Th(\nu(p_n))
\]
collapses everything outside $U$ to the base point. Let $\gamma_{d,n} \to \Gr_d(\R^{d+n})$ be the tautological $d$-dimensional vector bundle over the Grassmanian of oriented $d$-dimensional linear subspaces in $\R^{d+n}$ and let $(-\gamma_{d,n})$ be its $n$-dimensional complement. 
Composition with the classifying map of the normal bundle $\nu(p_n)$
\[
\Th(\nu(p_n))=\Th(-T_{p_n} E_n(M)) \to \Th(-\gamma_{d,n})=\MTSO(d)_{n+d}
\]
yields a composite map
\[
\Sigma^{n+d} (B_n(M))_{+} \to \Th(\nu(p_n)) \to \MTSO(d)_{n+d}.
\]
Letting $n \to \infty$ and passing to the adjoint map gives the desired map
\[
B \Diff(M)= B_{\infty}(M) \to \Omega^{\infty} \MTSO(d).
\]

\begin{lemma} \label{identify-alpha} Let $M$ be as above and consider the maps
	\[
\xymatrix{
\pi_0 \Diff(M) \cong \pi_1 B \Diff(M) \ar[r]^(0.65){\pi_1(\lambda_M)} & \pi_1 \Omega B \Cob_{d} \ar[r]^{\mathfrak{d}} \ar[d]^{\cong} & \pi_0 \Omega B \Cob_{d+1} \ar[r]^{\cong} \ar[d]^{\cong} & SKK_{d+1} \\    
 &   \pi_1 \MTSO(d) \ar[r]^(.45){\mathfrak{d}} & \pi_0 \MTSO(d+1) &
}
	\]
where $\mathfrak{d}$ is the connecting homomorphism. Then the upper composite map sends the class of $f \in \Diff(M)$ to the class of the mapping torus $[T_f]$.
\end{lemma}
\begin{proof}

Let $ f \in \Diff(M)$ and let $T_f$ be the associated mapping torus. We consider the corresponding smooth $M$-bundle $\pi_f \colon T_f \to S^1$, its classifying map $c_f \colon S^1 \to B\Diff(M)$ and a fiberwise embedding into $S^1 \times \R^{d+n}$. Note that under the identification $\pi_1 B\Diff(M) \cong \pi_0 \Diff(M)$ the loop $c_f$ exactly represents $f \in \Diff(M)$. Assuming that $n$ is large enough, there is a commutative diagram as follows
	\[ \begin{tikzcd}
	\nu_{\pi_f} \arrow[r] \arrow[d] & \nu_{p_{n}} \arrow[r] \arrow[d]& -\gamma_{d,n} \arrow[d]\\
	T_f
	\arrow{d}[swap]{\pi_f} \arrow[r] & E_{n}(M) \arrow{d}{p_n} \arrow[r] & \Gr_d(\R^{d+n})\\
	S^1 \arrow{r}{c_f} & B_{n}(M). &
	\end{tikzcd}
	\]
Then the Pontryagin-Thom construction applied to the $M$-bundle $\pi_f \colon T_f \to S^1$ produces a map 
	\[
	\Sigma^{d+n} (S^1)_{+} \to \Th(\nu_{\pi_f}).
	\]
After composing with the map to $\Th(-\gamma_{d,n})$ and letting $n \to \infty$ we obtain an element in $\pi_1(\MTSO(d))$. By the construction of the connecting homomorphism $\mathfrak{d}$, in order to identify the image of this element in $\pi_0 MTSO(d+1) \cong SKK_{d+1}$, we need to identify the pullback of the zero section $\Gr_d(\R^{n+k})$ along the composite map $\Sigma^{d+n}(S^1)_{+} \to \Th(-\gamma_{d,n})$, which is given by $T_f$:
	\[ \begin{tikzcd}
	\Th(\nu_{\pi_f}) \arrow[r] & \Th(\nu_{p_{n}}) \arrow[r] & \Th(-\gamma_{d,n}) \\
	\Sigma^{d+n} (S^1)_{+}
	\arrow{u}{PT_{\pi_f}} \arrow{r}{\Sigma^{d+n} c_f} &  \Sigma^{d+n} (B_{n})_{+} \arrow{u}{PT_{p_n}} & \\
	T_f \arrow[hookrightarrow]{u} \arrow[rr] & & \Gr_{d}(\R^{d+n}) \arrow[hookrightarrow]{uu}.
	\end{tikzcd}
	\]
This completes the identification of the upper composite map.	
\end{proof}

\begin{lemma}\label{alphaim} The image of $\alpha \colon \pi_2 B \End_d^{\varnothing} \to SKK_{d+1}$ (\autoref{isom}) is the subgroup generated by the classes of mapping tori. 
\end{lemma}
\begin{proof}
Let $\mM := \End_d^{\varnothing}=\bigsqcup_{N} B \Diff(N)
$ be the homotopy commutative monoid, where the disjoint union is over closed oriented smooth $d$-dimensional manifolds, one from each diffeomorphism class. By the group completion theorem \cite{McS_groupcompletion, RW_groupcompletion}, 
\[
H_*(\mM) [ \pi_0^{-1} ] \xrightarrow{\cong}  H_* (\Omega_{\varnothing} B \mM) 
\]
where $\pi_0=\pi_0 \mM$. The localization $H_*(\mM)[\pi_0^{-1}]$ can be expressed as the homology of the space
$$\mM_{\infty} = \mathrm{hocolim}(\mM_0  \xrightarrow{\bigsqcup N_1} \mM_1 \xrightarrow{\bigsqcup N_2} \mM_2 \to \cdots)$$
where each $\mM_i$ is $\mM$ and $N_1, N_2, \ldots$ is a sequence of closed $(d-1)$-manifolds such that for every $N$ and $j \geq 0$, there exists a $k \geq 0$ such that $N$ is a right factor of $N_{j+1} \bigsqcup N_{j+2} \bigsqcup \cdots \bigsqcup N_{j+k}$ in the discrete monoid $\pi_0(\mM)$ -- this can be arranged because $\pi_0$ is countable. For instance, a class $x \in H_1(\mM_1)$ represents the class $x \cdot N_1^{-1}$ in the localization, etc. Since
\[
H_1(\mM) = \bigoplus_{N} H_1(B\Diff(N)) \cong \bigoplus_{N} \pi_1(B\Diff(N))_{ab}  \cong \bigoplus_{N} \pi_0 \Diff(N)_{ab}
\]
it follows that every class in 
$$\pi_1 \Omega_{\varnothing} B \mM \cong  H_1 \big(\Omega_{\varnothing} B \mM\big)_0,$$
where $(-)_0$ denotes the summand corresponding to $[\varnothing]$, is represented by a class in 
$$H_1(B\Diff(N_1 \bigsqcup \cdots \bigsqcup N_k)) \subseteq H_1(\mM_k).$$
Recall that $\alpha$ corresponds to the connecting homomorphism of the homotopy fiber sequence \eqref{seq: Cobf Genauer type sequence}, and this factors through the connecting homomorphism $\mathfrak{d}$ of  \autoref{identify-alpha}. These observations show that the image of $\alpha$ is generated by the images for all $M$ of the upper composite maps displayed in  \autoref{identify-alpha}. Hence, applying  \autoref{identify-alpha}, we conclude that the image of 
$\alpha$ consists of the classes of mapping tori. 
\end{proof}

\begin{rem} 
While the entire homotopy type of $B\Cobf_d$ is not known in general, we can actually determine it for $d=2$.
    There is a diagram of infinite loop spaces
    $$
\xymatrix{
& \Omega B \Cob_2 \ar[d] \ar@{=}[r] & \Omega B \Cob_2 \simeq \Omega^{\infty} MTSO(2) \ar[d] \\
\Omega^{\infty} MTSO(2) \ar[r] \ar@{=}[d] & \Omega B \Cobf_2 \ar[r] \ar[d]^{\partial} &  \Omega B \Cob_2^\partial \simeq Q\mathbb{C}P^{\infty}_+ \ar[d]^{\partial} \\
\Omega^{\infty} MTSO(2) \ar[r] & Q \mathbb{C}P^{\infty}_+ \ar[r]^{trf} & \Omega QS^0
}$$
where both rows and columns are fiber sequences. The right column is the Genauer fiber sequence \cite{Genauer12, Steimle21}, and the bottom row is constructed by Giansiracusa \cite{Giansiracusa} and 
is equivalent to the Genauer fiber sequence. The middle column is the homotopy fiber sequence of \autoref{seq: Cobf Genauer type sequence}. Here we are using that $\Omega B \End_{1}^{\varnothing} \simeq Q \mathbb{C}P^{\infty}_+$, where $\End_{1}^{\varnothing}$ is the monoid of 1-dimensional oriented smooth closed manifolds.
Since the lower right square is a homotopy pullback, it follows that the middle bottom arrow admits a section up to homotopy, therefore there is a splitting of infinite loop spaces:
$$\Omega B\Cobf_2 \simeq Q \mathbb{C}P^{\infty}_+ \times \Omega^{\infty} MTSO(2).$$
More generally, using arguments similar to \cite[pp. 337-338]{RS14}, the map $\Omega B \Cobf_d \to \Omega B \Cob_d^{\partial} \simeq Q(BSO(d)_+)$ also admits a section up to homotopy for all $d$, so $Q(BSO(d)_+)$ splits off of $\Omega B \Cobf_d$ for all $d$.
\end{rem}

\section{$K_1$-classes of diffeomorphisms and the Kervaire semicharacteristic} \label{section: K_1}

\subsection{Diffeomorphisms and $K_1$-classes} Let $\Mnfld^{\partial, d}_{un}$ denote the category with squares of \emph{unoriented} smooth compact $d$-dimensional manifolds -- more generally, the notation `un' will indicate that we work here with unoriented manifolds throughout this section. 

In this section, we show that $K_1^\square(\Mnfld^{\partial, d}_{un, \Delta})$ is nontrivial and, moreover, that we can detect the Kervaire semicharacteristic under a canonical map $K_1^\square(\Mnfld^{\partial,d}_{un, \Delta})\to K_1(\Z)$. This is the $\pi_1$ level of a map of spectra, which lifts the Euler characteristic. Note that the Kervaire semicharacteristic is an SKK-invariant which is not SK-invariant. 

 Recall that for any Waldhausen category $\C$, any object $C$ with a weak equivalence $\phi\colon C\to C$ gives rise to an element $[C, \phi]\in K_1(\C)$. Given an unoriented closed smooth $d$-manifold $M$ and a diffeomorphism $\varphi \colon M \to M$, we construct an associated element $[M, \varphi]$ in $K_1$ of the scissors congruence $K$-theory of unoriented manifolds, which maps to the corresponding element $[M_+, \varphi_+]$ in $A_1(\ast):= \pi_1 A(\ast)$ under the map of \autoref{subsec:map_to_A-theory}.

\begin{con} \label{topcon} We consider the following sequence of maps 
\[
B \Diff(M) \xrightarrow[]{\lambda_M} \Omega B\Cobf_{d, un} \xrightarrow[]{f_d} K^\square(\Mnfld^{\partial,d}_{un, \Delta}),
\] 
where $M$ is a closed smooth $d$-manifold. Let $\varphi \colon M \to M$ be a diffeomorphism. Since $[\varphi] \in \pi_0 \Diff(M) \cong \pi_1 B \Diff(M)$, any such diffeomorphism $\varphi$ of $M$ yields an element $[M, \varphi] \in K_1^\square(\Mnfld^{\partial,d}_{un, \Delta})$ via $\lambda_M$.  The image of $[\varphi]$ under the composite homomorphism
\[
\pi_1 B \Diff(M) \xrightarrow[]{\pi_1(\lambda_M)} \pi_1 \Omega B\Cobf_{d, un} \xrightarrow[]{f_d} K_1^\square(\Mnfld^{\partial,d}_{un, \Delta}) \xrightarrow[]{g_d} \pi_1 A(BO(d)) \xrightarrow[]{BO(d) \to \ast} \pi_1 A(\ast),
\] 
agrees with $[M_+, \varphi_+]$ in $A_1(\ast) = \pi_1 A(\ast)$. By \autoref{prop: decomposing BM-map as fg}, the composition of $f_d$ with the map to $A(BO(d))$ agrees with the corresponding B\"okstedt-Madsen map.\end{con}

The singular chain complex functor $S \colon \mathcal{R}^{hf}(\ast)\to \ChPerf_\mathbb{Z}$ induces a canonical linearization map $\ell \colon A(\ast) \to K(\ChPerf_{\mathbb{Z}})$ (using the $T\sbt$-construction). Thus, the element $[M_+, \varphi_+] \in A_1(\ast)$ further determines an element $[S(M_+), S(\varphi_+)] \in K_1(\ChPerf_\mathbb{Z}).$ Tracing through the Gillet-Waldhausen equivalence $K(\ChPerf_\mathbb{Z})\simeq K(\mathbb{Z})$ and using the description of $K_1(\mathbb{Z})$ in terms of automorphisms, we obtain the following identification.

\begin{prop}
 The element $[S(M_+), S(\varphi_+)]$ in $K_1(\ChPerf_\mathbb{Z})$ maps to the class 
 \[
 [\bigoplus_{i  \textrm{ even}} H_i(M)/ \tors, \varphi_{\ast, even}]-[\bigoplus_{i\  odd} H_i(M)/ \tors, \varphi_{\ast, odd}] \in K_1(\Z).
 \] 
 \end{prop}
Using the determinant isomorphism $K_1(\Z) \to \Z^{\times} = \{\pm 1\}$, the latter element can be further identified with 
   \[
      \prod_i \det(\varphi_{\ast, i})^{(-1)^i} \in \{\pm 1\},
   \]
where $\varphi_{\ast, i}$ denotes the automorphism induced by  $\varphi$ on $H_i(M)/\tors$. Since we work in $\Z^{\times} = \{\pm 1\}$, the latter is simply $\prod_i \det(\varphi_{\ast, i}) \in \{\pm 1\}.$

\subsection{Detection of the Kervaire semicharacteristic}
We recall that the Kervaire semicharacteristic of a closed smooth $(2k+1)$-dimensional manifold $M$ is a $\Z/2\Z (\cong \{\pm 1\})$-invariant defined as
   $$\kappa(M)=\sum_{i = 0}^k \mathrm{rank} H_{2i}(M, \Z)  \mod 2.$$ 
We now prove the following theorem. 

\begin{thm}\label{K1-classes}
Let $M$ be a closed orientable smooth $(2k+1)$-dimensional manifold and $\varphi \colon M \to M$ an orientation-reversing diffeomorphism. Then under the composition of the maps considered above,
$$K_1^\square(\Mnfld^{\partial,d}_{un, \Delta})\to A_1(\ast) \xrightarrow[]{\pi_1(\ell)} K_1(\Z) \cong \Z/2\Z,$$ 
the element $[M,\varphi]$ maps to $\kappa(M)$.
\end{thm}

\begin{rem} Using \autoref{prop: decomposing BM-map as fg}, the Dwyer-Weiss-Williams index theorem \cite{DWW, RS14, RS17} and the well-known fact that the composite 
$$QS^0 \xrightarrow[]{\eta} A(\ast) \xrightarrow[]{\ell} K(\Z)$$
is a $\pi_1$-isomorphism, \autoref{K1-classes} and its proof below essentially determine $\pi_1$ of the canonical map $B\Diff(M) \to QS^0$ that is defined by the Becker-Gottlieb transfer. 
\end{rem}

\begin{proof}
To prove the theorem, we show that for any closed smooth $d$-dimensional manifold $M$ and diffeomorphism $\varphi \colon M \to M$, the following formula holds
\begin{equation}\label{deteqn}\det(\varphi_i)\det(\varphi_{d-i})=\det(\varphi_d)^{\mathrm{rank} H_{i}(M)},\end{equation} where $\varphi_i$ (resp. $\varphi^i$) denotes here the isomorphism induced on the $i$-th rational homology (resp. cohomology) group. From this formula, when $d=2k+1$, we deduce
$$\prod_{i=0}^{2k+1} \det(\varphi_i) = \prod_{i=0}^{k} \det(\varphi_d)^{\mathrm{rank} H_i(M)},$$ 
which yields the required result, since $\det(\varphi_d)=-1$ when $\varphi$ is orientation-reversing. (Also, for any odd-dimensional manifold $M$ and orientation-preserving diffeomorphism $\varphi \colon M \to M$, we see from this formula that the element $[M,\varphi]$ maps to the trivial element in $\Z/2\Z$.)

It remains to prove \autoref{deteqn}. From the naturality of the cap product the following diagram is commutative ($\Q$-coefficients)
\begin{center}
\begin{tikzcd}
H^i(M) \otimes H_d(M) \arrow[r, "\id \otimes \varphi_d"] \arrow[d, "\varphi^i \otimes \id"'] & H^i(M) \otimes H_d(M) \arrow[r, "\cap", "\cong"'] & H_{d-i}(M) \\
H^i(M) \otimes H_d(M) \arrow[rr, "\cap", "\cong"'] && H_{d-i}(M) \arrow[u, "\varphi_{d-i}"']
\end{tikzcd}
\end{center}
\noindent Hence, 
\[
\det\big( (\id \otimes \varphi_d) \circ ((\varphi^i) ^{-1} \otimes \id) \big) =\det (\varphi_{d-i}).
\]
Using that $\det(\varphi^i)=\det(\varphi_i)$ and the formula for the determinant of a tensor product of maps, we deduce
\[
\det (\varphi_d)^{\mathrm{rank} H^i(M)} \det(\varphi_i)^{-\mathrm{rank} H_d(M)}=\det (\varphi_{d-i}).
\]
Finally, using that $\mathrm{rank} H_d(M)=1$ and $\mathrm{rank} H^i(M)=\mathrm{rank} H_i(M)$, we obtain the desired formula
\[
\det (\varphi_i)  \det (\varphi_{d-i})= \det (\varphi_d)^{\mathrm{rank} H_i(M)}.
\]
\end{proof}

\section{Appendix}

In this appendix, we show that for the construction of the scissors congruence $K$-theory spectrum from \cite{WIT}, there is an analogous $K_1$ computation to that from \autoref{K1-classes}. In \cite{WIT}, we constructed a spectrum $K^\square(\Mnfld_d^{\partial})$, where the input category was the squares category of oriented manifolds with boundary. We can instead take the squares category of unoriented manifolds with boundary, and we obtain an analogous squares $K$-theory, which we denote by  $K^\square(\Mnfld_{d, un}^{\partial})$. All constructions and results from \cite{WIT} could have been stated for both the oriented and unoriented case. We chose to  work in the oriented case throughout \cite{WIT}, because all the constructions and results are more subtle when we need to keep track of the orientations. 

Just as in \cite[Thoerem 4.4.]{WIT}, the spectrum $K^\square(\Mnfld_{d, un}^{\partial})$ recovers the $SK$-group for unoriented manifolds with boundary on $\pi_0$ (in fact, this is even more straightforward for unoriented manifolds since we do not have to worry about preserving orientations).  A map of spectra $$K^\square(\Mnfld_{d, un}^{\partial})\to K(\Z),$$ which recovers the Euler characteristic on $\pi_0$, can be defined exactly the same  as the one from \cite[Theorem 5.1]{WIT}. We note that by construction this map factors through the linearization map $A(\ast) \xrightarrow[]{\ell} K(\Z)$.

The difference between the objects from the input categories of oriented (or unoriented) manifolds considered in \cite{WIT}, and the categories of manifolds with $\theta$-structure corresponding to oriented (or unoriented) manifolds taken as input in this paper, is 
whether the objects come equipped with a homotopy class of a classifying map to $BSO(d)$ (or $BO(d)$), or such a map on the nose -- but the space of such choices of $\theta$-structure, as encoded after thickening, is homotopically discrete in this case.

 We now show that for a diffeomorphism $\varphi\colon M\to M$, we can explicitly construct an element in $K_1^\square(\Mnfld_{d, un}^{\partial})$, directly from the squares construction, which maps to $[M_+, \varphi_+]$ in $A(\ast).$

 \begin{con}
Let $M$ be a closed smooth $d$-manifold and let $\varphi\colon M\to M$ be a diffeomorphism. Consider the following (2,0)-simplex and (1,1)-simplex in the bisimplicial set $N_{\sbt} T_{\sbt}  \Mnfldbd_{d,un}$
	$$ \varnothing \hookrightarrow M\xhookrightarrow{\varphi} M, \ \ \ \ \ \ \ \ \ \ \ \ \ \ \ \ \begin{tikzcd}[column sep=tiny, row sep=tiny]
    	\varnothing\arrow[rr, hook, "="] \arrow[dd, two heads] && \varnothing\arrow[dd,two heads]\\
    	& \square &\\
    	M\arrow[rr, hook, "\varphi"] && M.
    	\end{tikzcd}$$
 Note that these can be glued together to form a 2-sphere in $|N_{\sbt} T_{\sbt} \Mnfldbd_{d,un}|$, and thus represent an element in $K_1^\square(\Mnfldbd_{d,un})$, which we denote by $[M, \varphi]'$. The following diagram illustrates how we get the image of $S^2$ with its poles identified inside our space (we are careful not to glue together the map $\emptyset\to M$ viewed as a vertical map with itself viewed as a horizontal map, since we are not allowed to do that):
    \begin{center}
        \begin{tikzcd}[column sep=tiny, row sep=tiny]
        & \varnothing \arrow[ddr, hook] \arrow[ddl, hook']&\\
        &&&&\\
        M \arrow[rr, hook, "\varphi"] && M\\
        & \square &\\
        \varnothing \arrow[uu, two heads] \arrow[rr, hook, "="'] && \varnothing \arrow[uu, two heads].
    	\end{tikzcd}
	\end{center}
Let $R^{hf}(\ast)$ be the Waldhausen category of homotopy finite pointed spaces, and let  $R^{hf}(\ast)^\square$ be the associated category with squares. The  corresponding map (also denoted by) $g_d \colon \Mnfldbd_{d, un} \to R^{hf}(\ast)^\square$ 
is simply given on objects by $M\mapsto M_+$. The element $[M, \varphi]' \in K_1^\square(\Mnfldbd_{d,un})$ maps to an element in $K_1^\square(R^{hf}(\ast)^\square)$ formed by glueing together the (2,0)- and (1,1)-simplices
\begin{equation}\label{simplices} \ast \hookrightarrow M_+\xhookrightarrow{\varphi_+} M_+ \ \ \ \ \ \ \ \ \text{and} \ \ \ \ \ \ \ \ \begin{tikzcd}[column sep=tiny, row sep=tiny]
    	\ast\arrow[rr, hook, "="] \arrow[dd] && \ast\arrow[dd]\\
    	& \square &\\
    	M_+\arrow[rr, hook, "\varphi_+"'] && M_+.
    	\end{tikzcd}\end{equation}

In order to track where this element maps to under the isomorphism $K_1^\square(R^{hf}(\ast)^\square)\cong A_1(\ast)$, we need to consider the zig-zag from \cite[page 334]{waldhausen}
\[ \xymatrix{ wT\sbt\ \! R^{hf}(\ast) && wT^+\sbt\ \! R^{hf}(\ast) \ar[ll]_-\simeq \ar[rr]^-\simeq && wS\sbt \ \! R^{hf}(\ast).}\]

First, we lift our element in $K_1^\square(\Mnfldbd_{d, un})$ to $\pi_2(|wT^+\sbt\ \! \mathcal{R}^{hf}(\ast)|)$ along the map on the left. Recall that $T_n^+ \mathcal{R}^{hf}(\ast)$ is formed from $T_n \mathcal{R}^{hf}(\ast)$ by including choices of subquotients, and note that the (2,0)- and (1,1)-simplices in $N\sbt w T^+\sbt\ \! \mathcal{R}^{hf}(\ast)$
$$\begin{tikzcd}[column sep=small, row sep=small]
     	\ast\arrow[rr, hook] \arrow[dd, two heads] &&M_+ \arrow[dd, two heads] \arrow[rr, hook, "\varphi_+"] && M_+ \arrow[dd, two heads] \\
 	&&&&\\
 	\ast \arrow[rr, hook] && M+ \arrow[dd, two heads] \arrow[rr, hook, "\varphi_+"] && M_+ \arrow[dd, two heads] \\
 	&&&&\\
 	&& \ast \arrow[rr, hook] && \ast \ar[dd, two heads] \\
 	&&&&\\
 	&&&& \ast
     	\end{tikzcd}\ \ \ \ \ \ \ \  \text{and} \ \ \ \ \ \ \ \  \begin{tikzcd}[row sep=small, column sep = small]
     \ast \arrow[rr, hook, "="] \arrow[dr] \arrow[dd,swap] &&
     \ast \arrow[dd] \arrow[dr] \\
     & M_+ \arrow[rr, hook, crossing over, "\varphi_+"] &&
     M_+ \arrow[dd] \\
     \ast \arrow[rr, hook] \arrow[dr] && \ast \arrow[dr] \\
     & \ast \arrow[rr, hook] \ar[from=2-2,crossing over] && \ast
     \end{tikzcd}$$
\noindent map to the simplices from \autoref{simplices}. Next we consider their images under the righthand map  $w T^+\sbt\ \! \mathcal{R}^{hf}(\ast)\to w S\sbt \ \! \mathcal{R}^{hf}(\ast),$ which deletes the first row and retains the staircase diagram of subquotients. 
We observe that our simplices above map to the (2,0)- and (1,1)-simplices
 	$$\begin{tikzcd}[column sep=small, row sep=small]
 	\ast \arrow[rr, hook] && M+ \arrow[dd, two heads] \arrow[rr, hook, "\varphi_+"] && M_+ \arrow[dd, two heads] \\
 	&&&&\\
 	&& \ast \arrow[rr, hook] && \ast \ar[dd, two heads] \\
 	&&&&\\
 	&&&& \ast
     	\end{tikzcd}\ \ \ \ \ \ \ \  \text{and} \ \ \ \ \ \ \ \  \ast \xrightarrow{\sim} \ast,
 	$$
 where the second is degenerate, and the first (2,0)-simplex has faces given by $M_+, M_+$ and $\ast$. Collapsing the degenerate edge $\ast$ and glueing together the other edges, we obtain a $2$-sphere in $| wS\sbt \ \! \mathcal{R}^{hf}(\ast)|$, which depends on the diffeomorphism $\varphi$. This is precisely the element $[M_+, \varphi_+]$  in $A_1(\ast)$.

\end{con}

The proof of the following theorem is now identical to that of \autoref{K1-classes}.

\begin{thm}\label{K1-classes2}
Let $M$ be a closed orientable smooth $(2k+1)$-dimensional manifold and $\varphi \colon M \to M$ an orientation-reversing diffeomorphism. Then under the map
$$K_1^\square(\Mnfldbd_{d, un})\to  K_1(\Z) \cong \Z/2\Z,$$ 
the element $[M,\varphi]$ maps to $\kappa(M)$.
\end{thm}

 \bibliographystyle{alpha}
  \bibliography{biblio}

\begingroup%
\setlength{\parskip}{\storeparskip}

\end{document}